\documentclass{amsart}
\usepackage{graphicx} % Required for inserting images
\usepackage{amsmath}
\numberwithin{equation}{section}
\allowdisplaybreaks
\usepackage{amsthm}
\usepackage{amssymb}
\usepackage{enumitem}
\usepackage{xcolor}
\usepackage[colorlinks, linkcolor=blue, citecolor=red!75!blue]{hyperref}
\usepackage{verbatim}
\usepackage{xparse}
\usepackage[a4paper, total={6in, 8in}]{geometry}
\newtheorem{theorem}{Theorem}[section]
\newtheorem{corollary}[theorem]{Corollary}
\newtheorem{lemma}[theorem]{Lemma}
\newtheorem{proposition}[theorem]{Proposition}
\newtheorem{definition}[theorem]{Definition}
\newtheorem{example}[theorem]{Example}
\newtheorem{remark}[theorem]{Remark}
\newtheorem{conjecture}[theorem]{Conjecture}
\usepackage{amssymb}
\title{On the convoy of the ASEP speed process}
\author{Yuan Tian}
\date{\today}
\begin{document}
\begin{abstract}
    We investigate the size of the convoy in the speed process in the multi-species asymmetric simple exclusion process (ASEP). Through a coupling argument, we obtain an exact formula for the expected convoy size by relating it to a combinatorial structure. We prove that the asymptotic expected convoy size is universal for all fixed jump rates $q \in [0,1)$. In the special case $q=0$, we upgrade this to full convergence in distribution. We further establish a critical scaling \( q = 1 - \gamma/\sqrt{n} \) that yields a nontrivial limiting regime. Our analysis builds on Martin’s construction of the convoy and makes use of an orthogonal-polynomial representation of random-walk transition probabilities.
\end{abstract}
\maketitle

\tableofcontents
\section{Introduction}
\subsection{Overview}
The Asymmetric Simple Exclusion Process (ASEP) is one of the most well-studied models in the Kardar-Parisi-Zhang (KPZ) universality class \cite{PhysRevLett.56.889}. The system consists of particles and holes, and the particles can jump both to the right and to the left. If particles are allowed to jump in only one direction, the model reduces to the Totally Asymmetric Simple Exclusion Process (TASEP). First introduced in the probability literature by Spitzer in 1970 \cite{spitzer1970interaction}, it was subsequently studied by both mathematicians and physicists. For further background on interacting particle systems, see the standard reference by Liggett \cite{liggett1985interacting}.

There are various generalizations of the ASEP model. Among these, a particularly important extension is ASEP with a second-class particle, which provides a canonical probe for understanding microscopic characteristics of the rarefaction fan arising from step initial data. The speed of the second-class particle was shown to be uniformly distributed on $\left[-1,1\right]$ for TASEP in \cite{ferrari1995second, mountford2005motion, ferrari2005competition} and later for ASEP in \cite{aggarwal2023asep}. A further generalization of ASEP is the Multi-species ASEP. The Multi-species ASEP has infinitely many particle types, each of which can be regarded as a second-class particle after a suitable color projection. The speed of the second-class particle is also generalized to the \textit{ASEP speed process} \cite{aggarwal2023asep, amir2011tasep}, which is a collection of identically distributed but not independent uniform random variables. The \textit{convoy} is defined as the set of particles sharing the same speed in the Multi-species ASEP, which, despite the continuous marginal distribution of speeds, forms an almost surely infinite set.

In this work, we investigate the asymptotic size of the convoy. More specifically, we study how many particle labels from $1$ to $n$ have exactly the same speed as particle $0$ for large $n$ under two different settings for $q$. We first express the expectation of the size of the ASEP convoy explicitly via an intriguing combinatorial object known as the q-Genocchi number (Theorem \ref{thm.main}). The q-Genocchi number has appeared in earlier independent combinatorial studies, but its relation to probability was previously unknown. However, due to the ill-posed nature of the generating function of the q-Genocchi number, we are unable to carry out a complete asymptotic analysis. 

We obtain asymptotic results in a different approach. It turns out that the convoy size can be analyzed via Askey-Wilson polynomials with parameter $\left(a,b,c,d\right)=\left(1,0,0,0\right)$, which allows us to show that the expected size of the convoy grows on the order of $\sqrt{n}$. Moreover, we determine the constant after normalizing the expectation of the convoy size by $\sqrt{n}$ and prove that it is independent of the jumping rate $q \in [0,1)$ (Theorem \ref{thm.uni}). We conjecture that this universality also holds at the level of convergence in distribution, although we are only able to establish this in the special case $q = 0$ (Theorem \ref{thm.normal}). Finally, we examine a different scaling regime, namely, the double-scaling limit in which $q_n = 1 - \gamma/\sqrt{n} \to 1$ as $n \to \infty$. We identify the asymptotic behavior of the expected convoy size for fixed $\gamma \in \left(0,\infty\right)$ using the same family of orthogonal polynomials (Theorem \ref{weakly}). In this regime, letting \( \gamma \to \infty \) recovers the limiting size for any fixed \( q < 1 \), while letting \( \gamma \to 0 \) forces the convoy size to become negligible on the \( \sqrt{n} \) scale, showing that the convoy disappears as \( q \to 1 \) sufficiently quickly.

\subsection{Definitions and Main results}
In this paper, we consider multi-type interacting particle systems, meaning that some particles are stronger than others and therefore may jump over weaker particles. By convention, particles with smaller labels are stronger than those with larger labels.
\begin{definition}
    The multi-type Asymmetric Simple Exclusion Process (ASEP) $X\left(t\right)$ on $\mathbb{Z}$ is a random permutation of $\mathbb{Z}$ indexed by $t$, with initial condition $X_i\left(0\right) = i$. Here, $X_n\left(t\right)$ is the type of the particle at position $n$ at time $t$. 
    
    On each interval $\left[n, n+1\right]$, there is a Poisson clock with rate $1+q$ and a uniform random $U$ variable in $\left[0,1+q\right]$. The uniform random variable is refreshed each time the clock rings. When the clock rings at time $t$, if $X_n\left(t\right) < X_{n+1}\left(t\right)$ and $U \in [0,1]$, then $X_n\left(t\right)$ and $ X_{n+1}\left(t\right)$ swap. If $X_n\left(t\right) > X_{n+1}\left(t\right)$ and $U \in [1,1+q]$, $X_n\left(t\right)$ and $ X_{n+1}\left(t\right)$ also swap. 
    
    When $q = 0$, the model is called Totally multi-type Asymmetric Simple Exclusion Process (TASEP).
\end{definition}
\begin{remark}
    The existence of this particle system with infinitely many particles is guaranteed by the Harris construction \cite{harris1972nearest, harris1978additive}. 
\end{remark}
\begin{remark}
    There are different terms related to the system. The type of a particle might be also named as the species/color/label of the particle. We will use a mix of them in this paper. 
\end{remark}

Many other particle systems can be recovered from the multi-type ASEP under certain color projections. For instance, if we regard all particles with negative labels as first-class particles, all particles with positive labels as holes, and the particle with label $0$ as a second-class particle (stronger than holes but weaker than first-class particles), then we obtain the ASEP with step initial condition and a single second-class particle. We denote the resulting process by $Y\left(t\right)$. A well-known result concerns the trajectory of the second-class particle: it was first established for TASEP in \cite{mountford2005motion} and later extended to ASEP in \cite{aggarwal2023asep}.

\begin{theorem}[\cite{mountford2005motion}, Theorem 1,  \cite{aggarwal2023asep}, Theorem 1.1]
    Let $Y_0\left(t\right)$ denote the position of the second class particle at time $t$ in the ASEP with step initial condition and a single second-class particle. Then
    \[
    \lim_{t \to \infty}\frac{Y_0\left(t\right)}{\left(1-q\right)t} \overset{\text{a.s.}}{=} U,
    \]
    where $U$ is a uniform random variable on $\left[-1,1\right]$. We define $U$ to be the speed of the second-class particle.
\end{theorem}

For the multi-species ASEP, we can consider a family of color projections at each site. If we treat particle $i$ as the second class particle, projecting all stronger particles to first class particles and all weaker particles to holes, then we obtain the speed of particle $i$. As a corollary of the previous theorem, this construction yields a family of uniform random variables that are identically distributed but not independent, since they are coupled through the original multi-species system. This motivates the definition of the ASEP speed process. 

\begin{definition}[\cite{amir2011tasep}, Definition 1.3]
    Let $X\left(t\right)$ denote the multi-species ASEP, and let $X_n\left(t\right)$ be the position of the type-$n$ particle at time $t$. Then
    \[
    \lim_{t\to \infty}\frac{X_n\left(t\right)}{\left(1-q\right)t} \overset{\text{a.s.}}{=} U_n,
    \]
    where $\{U_n\}_{n \in \mathbb{Z}}$ is a family of uniform random variables on $\left[-1,1\right]$. This family is called the ASEP speed process.
\end{definition}

In \cite{amir2011tasep}, the authors proved several properties concerning the joint distribution of the speed process, including the exact joint distribution of $\left(U_0, U_1\right)$ and the fastest particle among $n$ consecutive particles. Some of these properties are surprising and peculiar; for example, the probability that the speeds $U_0$ and $U_1$ are equal is $\frac{1-q}{6}$, which is strictly positive. In this paper, we are particularly interested in one of the concepts arising from this speed process, namely, the convoy of the ASEP speed process (ASEP convoy).

\begin{definition}[\cite{amir2011tasep}, Theorem 1.8]
    The convoy of particle $0$ is the random set of indices corresponding to particles that have the same speed as particle $0$. Formally,
    \[
    \mathcal{C}_0 = \{ j \in \mathbb{Z} : U_j = U_0 \}.
    \]
    We also define the truncated convoys
    \[
    \mathcal{C}_0^n = \{ j \in [1,n] : U_j = U_0 \} \qquad \mathcal{C}_0^{\mathbb{Z}^+} = \{ j \in \mathbb{Z}^+ : U_j = U_0 \}.
    \]
\end{definition}

\begin{remark}
    All the concepts defined previously, including the configuration $X\left(t\right)$, the speed process $\{U_n\}_{n \in \mathbb{Z}}$ and the convoy sets $\mathcal{C}_0$ and $\mathcal{C}_0^n$, depend on parameter $q$. We omit the superscripts $q$ for simplicity but will write it explicitly when it is necessary, e.g. $\mathcal{C}_0^{\left(q\right)}$, $\mathcal{C}_0^{n,\left(q\right)}$.
\end{remark}

The first theorem that we prove in this work is an exact formula for the expectation of the size of the ASEP convoy. 

\begin{theorem}[Theorem \ref{thm.main}]\label{thm.1}
    Let $\# \mathcal{C}_0^n$ be the cardinality of set $\mathcal{C}_0^n$, and $E_{x}\left[A\right] := E\left[A|\frac{1+U_0} {2} = x\right]$. We have
    \begin{equation}
        \mathbb{E}_{x}\left[\# \mathcal{C}_0^n\right] = \sum_{k=0}^{n} \left(1-q\right)^{2k+1} \left(-x\left(1-x\right)\right)^{k+1} \binom{n+1}{k+1} B_k\left(1,q\right),
    \end{equation}
    where the $B_k\left(1,q\right)$ is the q-Genocchi number (see Definition \ref{Genocchi}).
\end{theorem}

The second main theorem of this paper concerns an asymptotic property of the ASEP convoy. It is worth mentioning that the appearance of a universal constant is surprising for the following reason. In \cite{amir2011tasep}, the authors showed by exact computation that, from a local perspective, smaller values of $q$ correspond to larger convoy sizes. This agrees with intuition: for smaller $q$, stronger particles exert a stronger blocking effect on weaker particles, making it more likely that several particles share the same velocity by being blocked by the same stronger particle. In contrast, we prove that when the number of particles with the same velocity is considered over a large interval, it grows on the order of $\sqrt{n}$ with a constant independent of $q$. In other words, the asymptotic size of the convoy is universal. 

\begin{theorem}[Theorem \ref{thm.uni}]\label{thm.2}
     The asymptotic expected size of the ASEP convoy of particle $0$ is of order $\sqrt{n}$ and is universal, that is, independent of $q$. More precisely,
    \begin{equation}
        \lim_{n \to \infty}\frac{1}{\sqrt{n}}\mathbb{E}_x\left[\# \mathcal{C}_0^n\right] = \sqrt{\frac{4x\left(1-x\right)}{\pi}}.
    \end{equation}
    Moreover, without conditioning on the speed of particle $0$, We have
    \begin{equation}
        \lim_{n \to \infty}\frac{1}{\sqrt{n}}\mathbb{E}\left[\# \mathcal{C}_0^n\right] = \frac{\sqrt{\pi}}{4}.
    \end{equation}
\end{theorem}

Not only does the normalized expected size of the convoy of ASEP converge to a constant, but in the special case $q = 0$ (corresponding to TASEP), the normalized convoy size, as a sequence of random variables, also converges in distribution to a folded normal random variable. This constitutes the third main result of our work.

\begin{theorem}[Theorem \ref{thm.normal}]\label{thm.3}
Consider TASEP ($q=0$). Conditionally on the event $U_0 = u$ with $x = \tfrac{1+u}{2}$, the asymptotic size of the convoy of particle $0$ satisfies
\begin{equation}
    \frac{\# \mathcal{C}_0^{n,\left(0\right)}}{\sqrt{n}}
    \Big|_{U_0=u} \xrightarrow[n \to \infty]{d}
    \bigl|\mathcal{N}\left(0,2x\left(1-x\right)\right)\bigr|,
\end{equation}
where $\mathcal{N}\left(0,2x\left(1-x\right)\right)$ denotes a normal random variable with mean $0$ and variance $2x\left(1-x\right)$. 
\end{theorem}

It is natural to generalize this previous result to the case of ASEP, i.e. the size of the convoy as a sequence of random variables  converges to the same distribution for varying $q$. However, we can not prove it within the method that we developed in this paper. We formulate it as a conjecture. 
\begin{conjecture}
     Conditionally on the event $U_0 = u$ with $x = \tfrac{1+u}{2}$, the size of the ASEP convoy is asymptotically a universal folded normal random variable. More specifically, we have
     \begin{equation}
    \frac{\# \mathcal{C}_0^{n,\left(q\right)}}{\sqrt{n}}
    \Big|_{ U_0 =u }\xrightarrow[n \to \infty]{d}
    \bigl|\mathcal{N}\left(0,2x\left(1-x\right)\right)\bigr|,
\end{equation}
\end{conjecture}
We further support the conjecture with numerical simulations. We run simulations for various $q$ for the particles with label in $\left[1,10000\right]$, for each $q$ we simulate 10000 times, and we set the speed of particle $0$ as $0$ (i.e. $x =\frac{1}{2}$) the histograms are the following.

 \begin{figure}[htbp]
        \centering
        \begin{minipage}{0.48\textwidth}
            \centering
            \includegraphics[width=\linewidth]{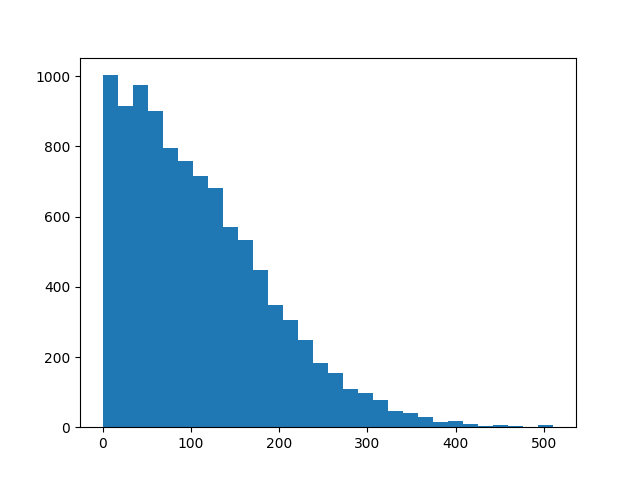}
        \caption{Size of ASEP convoy $q = 0$}
        \end{minipage}
        \begin{minipage}{0.48\textwidth}
            \centering
            \includegraphics[width=\linewidth]{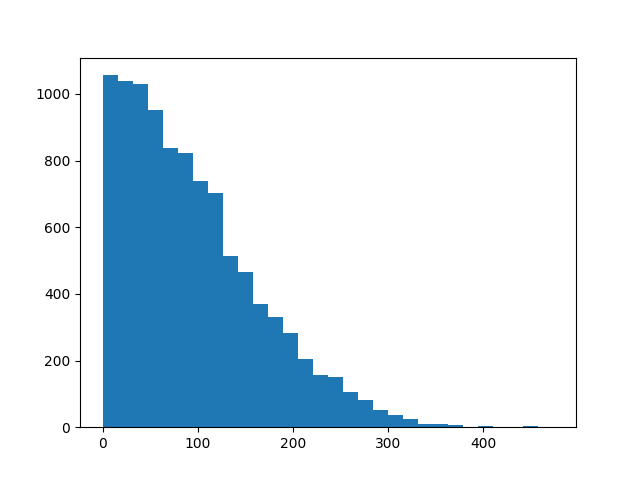}
        \caption{$q = 0.5$}
        \end{minipage}
    \end{figure}

In addition to the fixed-$q$ case, we also investigate the size of the ASEP convoy when $q \to 1$ as $n \to \infty$. The motivation for studying this scaling regime is the following. From our first result, convoy exists at scale $\sqrt{n}$ for any fixed $q \in [0,1)$. If we take $q = 1$, the process degenerates to a symmetric simple exclusion process, in which the interaction between particles disappears; hence the convoy does not exist when $q = 1$. Thus, it is natural to investigate the transition regime in which $q_n \to 1$ as $n \to \infty$. Letting $\gamma \to 0$ and $\gamma \to \infty$ respectively, our result implies that when $q \to 1$ sufficiently slowly, the expectation of convoy set has the same size as the fixed-$q$ case, whereas when $q \to 1$ fast, the convoy disappears at the scale of $\sqrt{n}$. 

\begin{theorem}[Theorem \ref{weakly}]
Let the speed of particle $0$ be $u$, and define
\[
c = \frac{1-u^{2}}{4}, 
\qquad 
q_{n} = e^{-\tfrac{\gamma}{\sqrt{n}}}.
\]
Then
\[
\lim_{n \to \infty}
\mathbb{E}\left[
    \frac{\# \mathcal{C}_0^{n,\left(q_n\right)}}{\sqrt{n}}
    \middle|
    U_0 = u
\right]
= \mathbb{E}\left[Y^\gamma - X^\gamma\right].
\]

$\left(X^\gamma,Y^\gamma\right)$ is a random vector with joint density
\begin{equation}\label{jointdensity}
f^{\gamma}_{X,Y}\left(x,y\right) = \frac{1}{2\pi}e^{-\gamma x}\int_{0}^{\infty}e^{-c w^{2}}\frac{|\Gamma\left(iw/\gamma\right)|^2}{|\Gamma\left(2iw/\gamma\right)|^2}\mathcal{J}\left(x,w\right)\mathcal{J}\left(y,w\right)dw,
\end{equation}
and the function $\mathcal{J}\left(z,u\right)$ is given by
\[
\mathcal{J}\left(z,w\right)
= 2\Re\left[
    \bigl(\gamma e^{\gamma z}\bigr)^{\tfrac{i w}{\gamma}}
    \frac{\Gamma\left(2iw/\gamma\right)}{\Gamma\left(iw/\gamma\right)}
    {}_1F_{1}\left(
        1 - \tfrac{iw}{\gamma};
        1 - \tfrac{2iw}{\gamma};
        -\tfrac{1}{\gamma} e^{-\gamma z}
    \right)
\right].
\]
\end{theorem}

\subsection{Related work}
The TASEP speed process was introduced and studied by Amir, Angel, and Valko in \cite{amir2011tasep}. The speed process was used to characterize all the ergodic stationary measures of TASEP with finitely many types of particles. The more general ASEP speed process was also considered in that work, although all the results were based on the assumption of the existence of the second class particle in the ASEP with step initial condition. This was not proved until recently by Aggarwal-Corvin-Ghosal \cite{aggarwal2023asep}, using methods from integrable probability. Speed processes have been also studied for other integrable models; see, for example \cite{amir2021tazrp, drillick2024stochastic}.

The convoy of the ASEP speed process was defined in the same paper \cite{amir2011tasep} as the speed process. The authors obtained a combinatorial description of the convoy set in terms of a renewal process. The combinatorial construction was later generalized to ASEP using a different approach based on a queuing process in \cite{martin2020stationary}. The proofs of the convoy construction rely on combinatorial results about the stationary distributions of finite type multi-species (T)ASEP \cite{ferrari2007stationary, martin2020stationary}, which are constructed via multi-line queuing processes. 

The connection between convoys and stationary distributions in finite multi-species ASEP arises from a fundamental symmetry property. This property, first studied in the context of understanding the algebraic origins of interacting particle systems \cite{amir2011tasep, angel2009oriented, borodin2021color}, was formalized in a unified framework using Hecke algebra by Bufetov \cite{bufetov2020interacting}. In particular, it was shown in \cite{bufetov2020interacting} that the distribution of a multi-type ASEP configuration coincides with that of its inverse permutation. This symmetry makes it possible to reinterpret probability events to describe the convoy in terms of the hydrodynamic behavior of second class particles. Indeed, by projecting the multi-type ASEP onto simpler systems with step initial conditions and applying well-known hydrodynamic limit results for ASEP \cite{rost1981non}, convoy probabilities can be expressed as the probability of observing a second class particle at a given position in a coupled system.

Despite the elegant combinatorial construction of convoys, research focusing on convoys themselves remains limited. We highlight two applications of convoys in related areas. Convoys appear in the study of mixing times for the open-boundary TASEP. In \cite{schmid2023mixing}, the author used the size of the TASEP convoy to provide intuition for the mixing time of the open boundary TASEP in the maximal current phase, showing that it is of order $N^{3/2}$. By the same heuristic argument presented in \cite[Section 1.3]{schmid2023mixing} our results also provide insight into other mixing time problems. Firstly, our result on the asymptotic size of ASEP convoy supports Conjecture 1.4 in \cite{schmid2023mixing}, which predicts that the mixing time of the asymmetric simple exclusion process in the maximal current phase is of order $N^{3/2}$. In addition, our result for the critical scaling $q = 1 - \gamma/\sqrt{n}$ can serve as intuition for the mixing time order $N^{3/2 + \kappa}$ for the open ASEP at the triple point \cite{ferrari2025mixing}, where the authors consider $q = 1 - n^{\kappa}$ for $\kappa \in [0,1/2]$. 

In a series of works \cite{busani2024diffusive, busani2022scaling, busani2024stationary}, the authors proved that the TASEP speed process converges weakly to the stationary horizon under an appropriate scaling. As a consequence, the convoys can be used to characterize the Busemann difference profiles of the directed landscape, a key universal object in the KPZ universality class introduced in \cite{dauvergne2022directed}. Later, \cite{aggarwal2024scaling} established the convergence of the ASEP speed process to the stationary horizon under a weaker topology as a corollary of their main theorem. The universality of the scaling limits of the ASEP speed process heuristically suggests that the size of the convoy set is likewise universal, although smaking this implication precise is technically subtle.

Our proofs, both for the expected size of the convoy for fixed $q\in [0,1)$ and for the critical scaling regime, build on orthogonal polynomials techniques. To the best of our knowledge, this connection between the convoy set and orthogonal polynomials has not previously appeared in the literature. However, the orthogonal polynomial methods have a long history in the study of models in the KPZ universality class, particularly in the context of the ASEP. In fact, the stationary measures of ASEP can be described via the Matrix Product Ansatz \cite{derrida1993exact}, which has motivated extensive subsequent work. Many different explicit solutions to the Ansatz have been developed \cite{derrida2004asymmetric, uchiyama2004asymmetric}, often involving representations in terms of tridiagonal matrices. The entries of these matrices are connected to the coefficients appearing in the three-term recurrence relations of various families of orthogonal polynomials, such as the Askey–Wilson polynomials \cite{askey1985some}. For further developments in this direction, see \cite{barraquand2023stationary, bryc2010askey, bryc2025limits, bryc2024limit, derrida2004asymmetric}.

The scaling regime that we consider for $q \to 1$ also appears in other studies of ASEP, see, for example, \cite{bertini1997stochastic, corwin2024stationary, corwin2018open, ferrari2025mixing} often referred to as the weakly asymmetric simple exclusion process (WASEP). Notably, a key breakthrough was achieved by Bertini and Giacomin \cite{bertini1997stochastic}, who proved that under weak asymmetry scaling the fluctuation fields converge to the stochastic Burgers equation, and through the Hopf–Cole transform to the KPZ equation. This result established a rigorous bridge between microscopic exclusion processes and macroscopic growth models in the KPZ universality class. More recent developments have focused on exclusion processes with boundaries. Corwin and Shen \cite{corwin2018open} analyzed the open ASEP in the weakly asymmetric regime and showed that it converges to the KPZ equation with boundary conditions. Extending this perspective, Corwin and Knizel \cite{corwin2024stationary} constructed stationary measures for the open KPZ equation, offering a probabilistic characterization of its equilibrium states. This was further studied in the mathematics and physics literature \cite{barraquand2021steady, bryc2023markov}, where the Laplace transform formula from \cite{corwin2024stationary} was inverted. 

The rest of the paper is organized as follows. In Section \ref{sec.martin} we introduce Martin’s construction for the multi-species ASEP, which characterizes the set of the convoy. In Section \ref{sec.prop}, we prove some basic properties of the inhomogeneous walk in Martin’s construction, based on which we derive the convergence in distribution of the size of TASEP convoy in Section \ref{sec.prop}, the exact formula for the expectation of the size of the ASEP convoy in Section \ref{exactformula}. Section \ref{sec.universal} is dedicated to the proof of the universality of the size of ASEP convoy for fixed $q \in [0,1)$. Section \ref{qto1} tackles the $q \to 1$ weakly asymmetric limit of the convoy size. 

\textbf{Acknowledgments.} The author would like to express his gratitude to Alexey Bufetov and Nikita Safonkin for numerous insightful discussions. 

\section{The Martin's Construction}\label{sec.martin}
In this section, we recall a recursive construction of the stationary distribution for the multi-species ASEP developed by Martin in \cite{martin2020stationary}, which plays a central role in our analysis.  

We use the following standard notation for the q-Pochhammer symbols throughout this paper:
\begin{equation*}
    \begin{aligned}
        \left(a;q\right)_n = \prod_{i=0}^{n-1} \left(1 - a q^i\right), \quad \quad \quad \left(a_1, \dots a_k;q\right)_n = \prod_{i=1}^{k} \left(a_i;q\right)_n\\
        \left(a;q\right)_\infty = \prod_{i=0}^{\infty} \left(1 - a q^i\right), \quad \quad \quad \left(a_1, \dots a_k;q\right)_\infty = \prod_{i=1}^{k} \left(a_i;q\right)_\infty
    \end{aligned}
\end{equation*}
We start by a well-known result on the stationary measure of the n-type ASEP.
\begin{theorem}[\cite{ferrari1991microscopic}, Lemma 3.3]
    Let $\lambda_1, \lambda_2, \dots, \lambda_n \geq 0$ be given such that 
    \[
    \lambda_1 + \lambda_2 + \cdots + \lambda_n < 1.
    \]
    Then there exists a unique translation-invariant stationary measure for the $n$-type ASEP on $\mathbb{Z}$, in which $\lambda_i$ is the density of particles of type $i$. Here, the density of type $i$ particles refers to the probability that a given site is occupied by a type $i$ particle under the stationary distribution. We denote this unique stationary measure by $\nu_{\lambda}^{\left(n\right)}$.
\end{theorem}  

Martin found an algorithm to generate $\nu_{\lambda}^{\left(n\right)}$ for any $n \geq 2$ by using multi-line queuing in \cite{martin2020stationary}. However, since we are concerned with the case of 2-type particles in this paper, for simplicity we restrict ourselves to the construction of $\nu_{\lambda,\mu - \lambda}^{\left(2\right)}$. 

We consider two $\mathbb{Z}$-indexed lattices, the arrival line $\{a\left(i\right)\}_{i \in \mathbb{Z}}$ and the service line $\{s\left(i\right)\}_{i \in \mathbb{Z}}$. Each site on the arrival (respectively, service) line is occupied with probability $\lambda$ (respectively, $\mu-\lambda$), and is a hole otherwise. Accordingly, we set $a\left(i\right)=1$ (respectively, $s\left(i\right)=1$) if site $i$ is occupied, and $a\left(i\right)=0$ (respectively, $s\left(i\right)=0$) otherwise. 

Martin’s construction defines a random algorithm which generates a new $\mathbb{Z}$-indexed lattice $\{d\left(i\right)\}_{i \in \mathbb{Z}}$ on $\mathbb{Z}$ from the arrival and the service line, called the departure line, where $d\left(i\right)=0$ denotes a hole, $d\left(i\right)=1$ a first-class particle, and $d\left(i\right)=2$ a second-class particle. 

In the following, we explain how this algorithm functions. Specifically, let 
\[
Q_i := \sup_{n \leq i-1} \left(\sum_{j = n}^{i-1}a\left(j\right) - \sum_{j= n}^{i-1}s\left(j\right)\right),
\]
which represents the number of arrivals waiting for services before position $i$. Since the mean $\lambda$ of the arrival line is less than the mean of the service line $mu$, by the law of large numbers, $Q_i$ is almost surely finite. We sometimes also call the value of $Q_i$ the length of the queue. Given $Q_i = k$, the algorithm for determining $d\left(i\right)$ is as follows:
\begin{itemize}
    \item If $a\left(i\right) = s\left(i\right) = 0$, then $d\left(i\right) = 0$ and $Q_i = k$,
    \item if $a\left(i\right) = s\left(i\right) = 1$, then the arrival immediately takes the service, so $d\left(i\right) = 1$ and $Q_i = k$,
    \item if $a\left(i\right) = 1, s\left(i\right) = 0$, then $d\left(i\right) = 0$ and $Q_i = k+1$, the arrival service is added to the waiting queue,
    \item if $a\left(i\right) = 0, s\left(i\right) = 1$, each particle in the queue independently rejects the service with probability $q$. Hence, with probability $q^k$, the service is rejected by all particles, producing a second-class particle at site $i$, so $d\left(i\right) = 2$, $Q_i = k$. With probability $1 - q^k$, the service is accepted by some particle in the queue, giving $Q_i = k -1 $ and $d\left(i\right) = 1$. 
\end{itemize}

\begin{remark}
    Clearly, all the notions $a\left(i\right)$, $s\left(i\right)$ and $d\left(i\right)$ (and some other that we will define later in this section) depend on parameter $\lambda$ and $\mu$ (later some of them also depend on $q$), we omit them when the meaning is clear from the context.
\end{remark} 

We can further write the recurrence relation for $Q_i$ in terms of the algorithm,
\begin{equation}
    Q_{i+1} = Q_{i} + \mathbf{1}_{a\left(i\right) = 1} - \mathbf{1}_{d\left(i\right) = 1}.\label{queue}
\end{equation}

Alternatively, from the update rules, $\left(Q_i\right)_{i\in\mathbb Z}$ evolves as a birth-death chain on $\mathbb Z_{\ge0}$: the queue increases by one when $a\left(i\right)=1,s\left(i\right)=0$ (probability $\lambda\left(1-\mu\right)$), decreases by one when $a\left(i\right)=0,s\left(i\right)=1$ and the service is accepted by some particle (probability $\mu\left(1-\lambda\right)\left(1-q^k\right)$), and otherwise remains unchanged. Thus
\[
    p_{k,k+1} = \lambda\left(1-\mu\right), \qquad 
    p_{k,k-1} = \mu\left(1-\lambda\right)\left(1-q^k\right), \qquad 
    p_{k,k} = 1-p_{k,k+1}-p_{k,k-1}.
\]

Its stationary distribution $\left(\eta_k\right)_{k\ge0}$ is determined by detailed balance and satisfies
\begin{equation}
    \eta_k
    = \left(\frac{\lambda}{1-\lambda}\frac{1-\mu}{\mu}\right)^k
      \prod_{j=1}^{k}\frac{1}{1-q^j}\eta_0,
    \label{initial2}
\end{equation}
with $\eta_0$ chosen so that $\sum_{k=0}^\infty \eta_k=1$. Existence of a stationary distribution follows from $\lambda<\mu$.

\begin{theorem}[\cite{martin2020stationary}, Theorem 3.1]
    Running the queue-length process $Q_i$ in its stationary distribution~\eqref{initial2}, and recording the corresponding outputs $d\left(i\right)$, produces the stationary distribution $\nu_{\lambda,\mu-\lambda}^{\left(2\right)}$ of the two-type ASEP on $\mathbb{Z}$. In this distribution, the density of first-class particles is $\lambda$, while the density of second-class particles is $\mu-\lambda$.
\end{theorem}

Since we are interested in the probability of $U_j$ in the convoy of $U_0$, the characterization of the convoy is taken on the conditioned probability space where there is a second-class particle at position. To motivate a theorem later from \cite{martin2020stationary}, we prove the following lemma. 
\begin{lemma}
    Given $d\left(0\right) = 2$, let $\pi_k$ denote the conditional distribution of $Q_1$. Then $\pi_k$ satisfies the recurrence
    \begin{equation}
        \pi_k = \frac{q}{1-q^k}\frac{\lambda}{1 - \lambda}\frac{1 - \mu}{\mu} \pi_{k-1}.\label{initial}
    \end{equation}
\end{lemma}
\begin{proof}
    By the fact that $\mathbb{P}\left(Q_1 = k \mid d\left(0\right)=2\right)=\mathbb{P}\left(Q_1 = k, Q_0 = k, d\left(0\right) = 2\right)$ and Bayes' rule
    \[
    \pi_k = \mathbb{P}\left(Q_1 = k \mid d\left(0\right)=2\right) = \frac{\mathbb{P}\left(Q_0 = k\right)}{\mathbb{P}\left(d\left(0\right) = 2\right)} \mathbb{P}\left(Q_1 = k, d\left(0\right) = 2 | Q_0 = k\right)= \frac{\eta_k q^k}{\mathbb{P}\left(d\left(0\right)=2\right)}.
    \]
    Similarly,
    \[
    \pi_{k-1} = \frac{\eta_{k-1} q^{k-1}}{\mathbb{P}\left(d\left(0\right)=2\right)}.
    \]
    Therefore, using \eqref{initial2}
    \[
    \pi_k = \frac{\eta_k q^k}{\mathbb{P}\left(d\left(0\right)=2\right)}
    = \frac{q}{1-q^k}\frac{\lambda}{1-\lambda}\frac{1-\mu}{\mu}\frac{\eta_{k-1} q^{k-1}}{\mathbb{P}\left(d\left(0\right)=2\right)} = \frac{q}{1-q^k}\frac{\lambda}{1-\lambda}\frac{1-\mu}{\mu}\pi_{k-1}.
    \]
\end{proof}

This lemma implies that, when studying properties of the process $Q_i$ under the condition that a second-class particle is present at position $0$, the distribution of $Q_1$ is not the stationary distribution $\eta$, but rather the conditional distribution $\pi$. 

Now we introduce a characterization of the convoy set that was essentially proved by Martin in \cite{martin2020stationary}. We fill some technical details which are omitted in \cite{martin2020stationary}. 
\begin{theorem}[\cite{martin2020stationary}, Proposition 6.4]\label{construction}
     Let $0 \leq \lambda < \mu \le 1$ be two real numbers. Consider an inhomogeneous random walk $\{Q^{\lambda,\mu}_n\}_{n \geq1}$, with initial distribution \(Q^{\lambda,\mu}_1 \sim \pi_k\) \eqref{initial}. We let \(\mathcal{U}^{\lambda,\mu}\) a random set of indices defined by the procedure described below. Conditional on \(Q^{\lambda,\mu}_i=k\), the joint dynamics of \(Q^{\lambda,\mu}_i\) and \(\mathcal{U}^{\lambda,\mu}\) evolve as follows:
    \begin{itemize}
        \item with probability $\lambda\left(1-\mu\right)$, $Q^{\lambda,\mu}_{i+1} = Q^{\lambda,\mu}_i + 1$, $i \notin \mathcal{U}^{\lambda,\mu}$,
        \item with probability $\lambda \mu + \left(1-\lambda\right)\left(1-\mu\right)$ , $Q^{\lambda,\mu}_{i+1} = Q^{\lambda,\mu}_i$, $i \notin \mathcal{U}^{\lambda,\mu}$,
        \item with probability $\left(1-\lambda\right)\mu q^k$ , $Q^{\lambda,\mu}_{i+1} = Q^{\lambda,\mu}_i$, $i \in \mathcal{U}^{\lambda,\mu}$,
        \item with probability $\left(1-\lambda\right)\mu\left(1-q^k\right)$ , $Q^{\lambda,\mu}_{i+1} = Q^{\lambda,\mu}_i - 1$, $i \notin \mathcal{U}^{\lambda,\mu}$.
    \end{itemize}
    
    Let \(U_0 = u\) be the speed of particle \(0\), uniformly distributed on \(\left[-1,1\right]\). Define
    \[
    x = \frac{1+U_0}{2},
    \]
    so that \(x\) is uniform on \(\left[0,1\right]\).
    Let $\mathcal{U}_n^{\lambda,\mu} := \#\{j \in [1,n]: j \in \mathcal{U}^{\lambda,\mu} \}$ be the size of $\mathcal{U}^{\lambda,\mu}$ in $\left[1,n\right]$. The limits of $\mathcal{U}^{\lambda,\mu}$ and $\mathcal{U}^{\lambda,\mu}_n$ exist when $\lambda \uparrow x$ and $\mu \downarrow x$. Denote the limits by $\mathcal{U}^{x,x}$ and $\mathcal{U}^{x,x}_n$ respectively. The convoy set of particle $0$ in the ASEP speed process $\mathcal{C}_0$ is distributed as the random set $\mathcal{U}^{x,x}$. More specifically, we have
    \begin{equation}
        \mathcal{U}^{x,x} \overset{\left(d\right)}{=} \mathcal{C}_0^{\mathbb{Z}^+}, \qquad \mathcal{U}_n^{x,x} \overset{\left(d\right)}{=} \mathcal{C}_0^{n}. \label{randomu} 
    \end{equation}
\end{theorem}
\begin{remark}
    There are two minor differences between the statement we present here and the one in Martin's original paper. 
    
    Firstly, Martin studied another family $W_i$, which is the limit of the $L$ type multi-ASEP on a ring as $L \to \infty$ and we study the ASEP speed process $U_i$. From the perspective of symmetry \cite{bufetov2020interacting}
    \[
    \{W_i\}_{i \in \mathbb{Z}} \overset{d}{=}\{\frac{1-U_i}{2}\}_{i \in \mathbb{Z}}.
    \]
    In particular, the convoy set of $W_i$ and the convoy set of $U_i$ have the same distribution. Hence, Martin's characterization for the convoy set of $W_i$ can be also regarded as a characterization for our $U_i$. 

    Secondly, Martin only claimed that our statement is apparent from his proof of a weaker result. For the self-containess, we provide a proof of the result established by Martin. Our proof is for $\{W_i\}_{i \in \mathbb{Z}}$. As we already explained in the first part of this remark, this is equivalent to $\{U_i\}_{i \in \mathbb{Z}}$.
\end{remark}

\begin{proof}
    In Martin's work \cite[Lemma 6.6]{martin2020stationary}, he proved that conditional on the event $\{\lambda \le W_0 < \mu \}$
    \[
    \{j>0: \lambda \le W_j < \mu \}\overset{d}{=} \mathcal{U}^{\lambda,\mu}.
    \]
    It suffices to establish the following two limits as $\lambda\uparrow x$ and $\mu\downarrow x$:
    \begin{equation}
    \mathcal{U}^{\lambda,\mu}\ \Rightarrow\ \mathcal{U}^{x,x}
    \end{equation}
    and, writing $I=\left(\lambda,\mu\right)$, conditional on $W_0\in I$ the law of $\{i>0:W_i\in I\}$ converges to the conditional law $\{i>0:W_i=W_0\}\big|W_0=x\}$.

    We first prove the convergence of $\mathcal{U}^{\lambda,\mu}$ to $\mathcal{U}^{x,x}$. Let $\lambda_n=x-\tfrac1n$ and $\mu_n=x+\tfrac1n$, and define $U_n:=\mathcal{U}^{\lambda_n,\mu_n}$. By the stochastic monotonicity argument used by Martin \cite[Lemma 6.5]{martin2020stationary}, for $\left(\lambda_m,\mu_m\right)\supset\left(\lambda_n,\mu_n\right)\supset\left(x,x\right)$ there is a coupling under which
    \[
    U_m\ \supseteq\ U_n\ \supseteq\ U_{x,x}\qquad\text{for all }m<n.
    \]
    Thus $U_1\supseteq U_2\supseteq\cdots\supseteq U_{x,x}$ almost surely; write $U_\infty:=\bigcap_{n\ge1}U_n$. We claim $U_\infty=U_{x,x}$ almost surely under the coupling, which implies $U_n\Rightarrow U_{x,x}$.

    By the construction of $U_\infty$, $U_\infty \supseteq U^{x,x}$. Conversely, we assume $i \in \mathbb{Z}^{+}, i \notin U^{x,x}$. By the coupling, this event depends on finitely many random variables on $\left[0,1\right]$ used in the coupling. Hence, we can find sufficiently large $n$, such that $i \notin U_n$ with probability one (by taking these random variable locating in the segments such that $Q^{x-\frac{1}{n}, x+\frac{1}{n}}$ has exactly the same path as $Q^{x,x}$). 

    Fix $x\in\left(0,1\right)$ and intervals $I_n=\left(x-\tfrac1n,x+\tfrac1n\right)$. Let $J\subset\mathbb Z_+$ be finite and let $\psi:\{0,1\}^J\to\mathbb R$ be a bounded Lipschitz function. We show that
    \[
    \mathbb E\Big[\psi\big(\mathbf 1_{\{W_j\in I_n\}}:j\in J\big)\Big|W_0\in I_n\Big]\longrightarrow\mathbb E\Big[\psi\big(\mathbf 1_{\{W_j=W_0\}}:j\in J\big)\Big|W_0=x\Big],
    \]
    which implies convergence in distribution of the conditional laws on $\{0,1\}^{\mathbb Z_+}$.

    Denote
    \[
    X_n=\psi\left(\mathbf 1_{\{W_j\in I_n\}}:j\in J\right),\qquad X^*=\psi\left(\mathbf 1_{\{W_j=W_0\}}:j\in J\right).
    \]
    By conditioning on the value of $W_0$, 
    \[
    \mathbb E\left[X_n|W_0\in I_n\right] =\frac{1}{|I_n|}\int_{I_n}\mathbb E\left[X_n|W_0=t\right]dt.
    \]
    By the Lebesgue differentiation theorem, for almost every $x$,
    \[
    \frac{1}{|I_n|}\int_{I_n}E\left[X^*|W_0=t\right]dt\longrightarrow\mathbb E\left[X^*|W_0=x\right].
    \]
    It remains to show that 
    \begin{equation}
        \label{toshow}\frac{1}{|I_n|}\int_{I_n}\big|\mathbb E\left[X_n|W_0=t\right]-\mathbb E\left[X^*|W_0=t\right]\big|dt\to0.
    \end{equation}
    For each $t$, the vectors $\left(\mathbf 1_{\{W_j\in I_n\}}\right)_{j\in J}$ and $\left(\mathbf 1_{\{W_j=t\}}\right)_{j\in J}$ differ only if $W_j\in I_n$ but $W_j\ne t$ for some $j$. By Lipschitzness of $\psi$,
    \[
    \big|\mathbb E\left[X_n|W_0=t\right]-\mathbb E\left[X^*|W_0=t\right]\big|\le L_\psi\sum_{j\in J}\mathbb P\big(W_j\in I_n,\ W_j\ne t\big|W_0=t\big).
    \]
    The joint law of $\left(W_0,W_j\right)$ has a bounded density off the diagonal (this is clear from Martin's paper, see the proof of the joint density of $W_0$ and $W_1$ \cite[page 47]{martin2020stationary} for example). Hence
    \[
    \mathbb P\left(W_j\in I_n,\ W_j\ne W_0|W_0\in I_n\right)=O\left(|I_n|\right),
    \]
    which implies \eqref{toshow}. 
\end{proof}
\begin{remark}
    We omit the superscript $x,x$ in related objects such as $Q_i^{x,x}$, $\mathcal{U}^{x,x}$, when the choice of parameter is clear from the context. 
\end{remark}
The main goal of this article is to investigate the asymptotic size of the convoy of $U_0$. By using the previous theorem, we turn this question to the study of random set $\mathcal{U}$. We end this section by proving a technical proposition that we use later in this paper, which gives another expression of the distribution of $Q_1$. 

\begin{proposition}\label{prop.finite}
    Let $\pi_k$ be the initial distribution of $Q_1$ in the case of \eqref{initial} after taking the limit $\lambda \uparrow x, \mu \downarrow x$. We have
    \begin{equation}
        \pi_k = \frac{\left(q;q\right)_{\infty}}{\left(q;q\right)_{k}} q^k.\label{lawofpi}
    \end{equation}
    In addition, $Q_1$ has finite mean for any fixed $q \in [0,1)$ 
    \begin{equation}
        \mathbb{E}\left[Q_1\right] < \infty.\label{Eq1}
    \end{equation}
\end{proposition}
\begin{proof}
    We compute
    \[
    \frac{\pi_k}{\pi_{k-1}} = \left(\frac{\left(q;q\right)_\infty}{\left(q;q\right)_k} q^k\right)/\left(\frac{\left(q;q\right)_\infty}{\left(q;q\right)_{k-1}} q^{k-1}\right) = \frac{q}{1-q^k},
    \]
    so the recurrence $\pi_k=\frac{q}{1-q^k}\pi_{k-1}$ holds.
    It remains to check normalization, we have
    \[
    \begin{aligned}
        \sum_{i=0}^{\infty} \pi_k &= \sum_{i=0}^{\infty}  \frac{\left(q;q\right)_{\infty}}{\left(q;q\right)_{k}} q^k = \left(q;q\right)_{\infty} \left(1 + \frac{q}{1-q} + \frac{q^2}{\left(1-q\right)\left(1 - q^2\right)} \dots\right)\\
        &= \frac{\left(q;q\right)_{\infty}}{\left(q;q\right)_{k}} q^k = \left(q;q\right)_{\infty} \left(\frac{1}{1-q} + \frac{q^2}{\left(1-q\right)\left(1 - q^2\right)} \dots\right) = \frac{\left(q;q\right)_{\infty}}{\left(q;q\right)_{\infty}} = 1.
    \end{aligned}
    \]
    
    For the $\mathbb{E}\left[Q_1\right]$, we use the trivial bound $\left(q;q\right)_{\infty} < \left(q;q\right)_{k}$,
    \[
    \mathbb{E}\left[Q_1\right] = \sum_{i=0}^{\infty} k \pi_k = \sum_{i=0}^{\infty} k \frac{\left(q;q\right)_{\infty}}{\left(q;q\right)_{k}} q^k < \sum_{i=0}^{\infty} k q^k = \frac{q}{\left(1-q\right)^2} < \infty.
    \]
\end{proof}

\section{Properties of the inhomogeneous random walk and first applications}\label{sec.prop}
\subsection{Properties of the inhomogeneous random walk}
In this section we study the properties of $Q_i$ in preparation for the analysis that follows. We first prove a symmetric property of the random walk $Q_i$, which is of independent interest. 

\begin{theorem}
Let $0<q<1$ and $x\in[0,1]$. Consider the Markov chain $\left(Q_i\right)_{i\ge1}$ on $\mathbb{Z}_{\ge0}$ with initial law
\[
\mathbb{P}\left(Q_1=k\right)=\pi_k=\frac{\left(q;q\right)_\infty}{\left(q;q\right)_k}q^k,
\]
and one-step transition probabilities (independent of $i$), for $k\ge0$,
    \[
    P\left(k,k+1\right)=x\left(1-x\right),\qquad
    P\left(k,k-1\right)=x\left(1-x\right)\left(1-q^k\right),\qquad
    P\left(k,k\right)=1-x\left(1-x\right)\bigl(2-q^k\bigr),
    \]
with the convention $P\left(0,-1\right)=0$. We write
    \[
    P\left(a,b\right):=\mathbb{P}\bigl(Q_{i+1}=b\big|Q_i=a\bigr)
    \]
    for the one-step transition probability, which does not depend on $i$. 
    Fix $n\ge1$ and a sequence $s_1,\dots,s_n\in\{-1,0,1\}$, and set $s:=\sum_{i=1}^n s_i$. Then
        \begin{equation}    \label{eq:main}\sum_{k\ge0}\mathbb{P}\left(Q_1=k,Q_2=k+s_1,\dots,Q_{n+1}=k+s\right) = q^{-s}\sum_{m\ge0}\mathbb{P}\left(Q_1=m,Q_2=m-s_n,\dots,Q_{n+1}=m-s\right).
    \end{equation}
\end{theorem}

\begin{remark}
    In the subsequent computation, the transition probability takes the form 
    \[
    1 - q^{k+\sum_{i=1}^{k} s_i},
    \]
    where the exponent \(k+\sum_{i=1}^{k} s_i\) could, a priori, be negative. By definition of \(Q_i\), however, the transition probability is only defined on $\mathbb{Z}^{+}$. One might therefore ask whether additional restrictions on the path \(\left(s_i\right)\) are required to guarantee positivity of \(Q_i\). Such restrictions are unnecessary: we extend the definition of weighted paths to the case when \(Q_i\) becomes negative. Indeed, whenever the path crosses the \(x\)-axis, the product acquires a factor $1 - q^{0} = 0$ ,which forces the total weight of that path to vanish. Consequently, paths with negative values of \(Q_i\) make no contribution to the overall sum. 
\end{remark}

\begin{proof}
We first prove a $q$--skew detailed balance identity. For all $k\ge 1$ and $\ell\in\{k-1,k,k+1\}$ one has
\begin{equation}\label{eq:skew-db}
\pi_kP\left(k,\ell\right) = q^{k-\ell}\pi_\ell P\left(\ell,k\right).
\end{equation}
Indeed, using $\left(q;q\right)_{k+1}=\left(1-q^{k+1}\right)\left(q;q\right)_k$ we obtain
\[
\frac{\pi_{k+1}}{\pi_k}=\frac{q}{1-q^{k+1}}.
\]
If $\ell=k+1$, then
\[
\pi_k P\left(k,k+1\right)=\pi_k x\left(1-x\right)
= q^{-1}\pi_{k+1}\bigl[x\left(1-x\right)\left(1-q^{k+1}\right)\bigr]
= q^{-1}\pi_{k+1}P\left(k+1,k\right).
\]
If $\ell=k-1$, then
\[
\pi_k P\left(k,k-1\right)=\pi_k x\left(1-x\right)\left(1-q^k\right)
= q\pi_{k-1}x\left(1-x\right)=q\pi_{k-1}P\left(k-1,k\right).
\]
The case $\ell=k$ is immediate. This proves \eqref{eq:skew-db}.

Now consider any path $x_0,\dots,x_n$ such that $|x_i-x_{i-1}|\le1$. Multiplying \eqref{eq:skew-db} along the edges of the path yields
\[
\pi_{x_0}\prod_{i=1}^n P\left(x_{i-1},x_i\right)
= \Biggl(\prod_{i=1}^n q^{x_{i-1}-x_i}\Biggr)
\pi_{x_n}\prod_{i=1}^n P\left(x_i,x_{i-1}\right)
= q^{x_0-x_n}\pi_{x_n}\prod_{i=1}^n P\left(x_i,x_{i-1}\right).
\]
Equivalently, if we let $x_0=k$ and $x_i=k+\sum_{j=1}^i s_j$, we have $x_n=k+s$ and therefore
\[
\mathbb{P}\left(Q_1=k,Q_2=k+s_1,\dots,Q_{n+1}=k+s\right)
= q^{-s}\mathbb{P}\left(Q_1=k+s,Q_2=k+s-s_n,\dots,Q_{n+1}=k\right).
\]
Summing over $k\ge0$ and letting $m=k+s$ on the right-hand side gives
\begin{equation}
\sum_{k\ge0}\mathbb{P}\left(Q_1=k,Q_2=k+s_1,\dots,Q_{n+1}=k+s\right) = q^{-s}\sum_{m\ge0}\mathbb{P}\left(Q_1=m, Q_2=m-s_n,\dots,Q_{n+1}=m-s\right).
\end{equation}
Renaming $m$ back to $k$ yields the desired identity \eqref{eq:main}. Any path that would visit a negative state has zero probability, so the index shift is valid without further conditions. This completes the proof.
\end{proof}

We introduce a coupling lemma that allows us to express the random variable $\mathcal{U}_n$ in a more tractable form. Fix $\frac{1+U_0}{2} = x$. Our goal is to compute the conditional expectation $\mathbb{E}\left[\mathcal{U}_n | U_0 = x\right]$, which we will denote concisely by $\mathbb{E}_x\left[\mathcal{U}_n\right]$. 

We define a homogeneous random walk $\{P_n\}_{n\geq 1}$ with initial condition
\[
P_1 = Q_1,
\]
and transition probabilities
\[
\begin{cases}
P_{i+1} = P_i + 1 & \text{with prob. } x\left(1-x\right),\\[6pt]
P_{i+1} = P_i     & \text{with prob. } x^2+\left(1-x\right)^2,\\[6pt]
P_{i+1} = P_i - 1 & \text{with prob. } x\left(1-x\right).
\end{cases}
\]

We couple the processes $Q$ and $P$ using the same family of i.i.d.\ random variables $\{R_i\}_{i\ge 1}$, where each $R_i \sim \mathrm{Unif}\left[0,1\right]$. The joint evolution is as follows:

\begin{itemize}
    \item If $R_i < x\left(1-x\right)$, then both chains move up:
    \[
    P_{i+1} = P_i+1, \qquad Q_{i+1} = Q_i+1.
    \]

    \item If $R_i \in \left(x\left(1-x\right), 1-x\left(1-x\right)\right)$, then both chains remain unchanged:
    \[
    P_{i+1} = P_i, \qquad Q_{i+1} = Q_i.
    \]

    \item If $R_i > 1-x\left(1-x\right)$, then $P$ always steps down:
    \[
    P_{i+1} = P_i - 1,
    \]
    while $Q$ behaves differently:
    \[
    Q_{i+1} =
    \begin{cases}
    Q_i, & \text{with prob. } q^{Q_i}, \\
    Q_i - 1, & \text{with prob. } 1-q^{Q_i}.
    \end{cases}
    \]
\end{itemize}

In this last case, if $Q$ refuses to step down (i.e.\ $Q_{i+1}=Q_i$), then $i \in \mathcal U$. Thus the random set $\mathcal U$ records exactly the indices at which a discrepancy between the coupled chains occurs.
\begin{lemma}
    We have:
\begin{enumerate}
    \item For all $n\ge 1$,
    \begin{equation}
        \mathcal U_n = Q_{n+1} - P_{n+1} \quad \text{almost surely}.
    \end{equation}

    In particular,
    \begin{equation}\label{coupling}
        \mathcal U_n \overset{\left(d\right)}{=} Q_{n+1} - P_{n+1}.
    \end{equation}
    \item Taking expectations,
    \begin{equation}
        \mathbb E_x\left[\mathcal U_n\right] = \mathbb E_x\left[Q_{n+1} - Q_1\right].\label{Ecoupling}
    \end{equation}
\end{enumerate}
\end{lemma}

\begin{proof}
    \eqref{coupling} is direct from the construction of the coupling argument. 
    
    For the expectation \eqref{Ecoupling}, note that $P$ is a homogeneous random walk with the zero mean of the increments. We get
    \[
    \mathbb E\left[\mathcal U_n\right] = \mathbb E\left[Q_{n+1}-P_{n+1}\right] = \mathbb E\left[Q_{n+1}\right] - \mathbb E\left[P_1\right] = \mathbb E\left[Q_{n+1}\right] - \mathbb E\left[Q_1\right],
    \]
    which is the claimed identity. 
\end{proof}

The coupling enjoys a nice property in the TASEP case. 
\begin{lemma}
    When $q = 0$, we have 
    \begin{equation}
        Q_n \overset{\left(d\right)}{=} \Bigl|P_n + \frac{1}{2}\bigr| - \frac{1}{2}.\label{couple2}
    \end{equation}
\end{lemma}
\begin{proof}
    The proof is based on a lumping of Markov chains. Partition \(\mathbb{Z}\) into pairs \(B_k:=\{k,-1-k\}\) for \(k\in\mathbb{Z}_{\ge 0}\), and define the folding map
    \[
    f:\mathbb{Z}\to\mathbb{Z}_{\ge 0},\qquad f\left(p\right):=\left|p+\tfrac{1}{2}\right|-\tfrac{1}{2}.
    \]
    We show that \(\widetilde Q_n:=f\left(P_n\right)\) is a Markov chain with the same transition probabilities as \(\left(Q_n\right)\). This can be done by a direct computation, for simplicity, we only show the case when $\widetilde Q_n = 0$.

    If $\widetilde Q_n = 0$, then \(P_n\in\{0,-1\}\) and \(P_{n+1}\) takes values in \(\{1,0,-1\}\) or \(\{0,-1,-2\}\), respectively, with probabilities \(x\left(1-x\right),x^2+\left(1-x\right)^2,x\left(1-x\right)\). In both situations, applying \(f\) shows that
    \[
    \mathbb{P}\left(\widetilde Q_{n+1}=1\mid \widetilde Q_n=0\right)=x\left(1-x\right),\qquad \mathbb{P}\left(\widetilde Q_{n+1}=0\mid \widetilde Q_n=0\right)=1-x\left(1-x\right).
    \]
\end{proof}
\subsection{Asymptotic size of the TASEP convoy}
In this subsection we focus on the TASEP convoy, where stronger results are available. Since our interest is restricted to the TASEP setting, we henceforth fix $q=0$ (only in this subsection). From this section to simplify the notation, we conventionally denote the speed of particle $0$ by $u$, $x = \frac{1+u}{2}$, and $c = x\left(1-x\right)$. 
\begin{theorem}\label{thm.normal}
    Given the speed of particle, the number of particles $\left[1, n \right]$ in the $\mathcal{C}_0$ of TASEP is asymptotically a folded normal random variable, more specifically,
    \begin{equation}
        \lim_{n \to \infty}\frac{\# \mathcal{C}_0^n}{\sqrt{n}} \overset{\left(d\right)}{=} |\mathcal{N}\left(0,2x\left(1-x\right)\right)|.\label{tasepconvoy}
    \end{equation}
\end{theorem}
\begin{proof}
    We define 
    \[
    D_n:=\#\{0\le k<n:\ P_k=0,\ P_{k+1}=-1\},
    \] as the number of down-cross from $0$ to $-1$ in $P_n$ in the first $n$ steps, and similarly
    \[
    E_n:=\#\{0\le k<n:\ P_k=-1,\ P_{k+1}=0\}.
    \]
    From the lumping argument \eqref{couple2}, having a particle $i$ in the convoy is equivalent to $Q_i = Q_{i+1} = 0$ and the down-cross of $Q_i$ is blocked. This is the union the $\{P_i = 0, P_{i+1} =-1\}$ and $\{P_i = -1, P_{i+1} = 0\}$ in terms of random walk $P_n$.
    Thus, 
    \[
    \# \mathcal{C}_0^n = D_n+E_n.
    \]
    To study the properties of $D_n$ and $E_n$, we introduce local time of a discrete random walk. For $n \geq 1$, we denote by 
    \[
    L_n\left(0\right) := \sum_{k=0}^{n-1} \mathbf{1}_{\{P_k = 0\}}
    \]
    the local time at $0$ up to time $n$, i.e., the total number of visits of the random walk to the origin before step $n$. 
    Denote by $\mathcal F_k$ the $\sigma$-algebra $\sigma\left(P_0,\dots,P_k\right)$. By direct computation, we have $\mathbb E\left[D_n - cL_n\left(0\right)\mid \mathcal F_{n-1}\right] = D_{n-1} - cL_{n-1}\left(0\right)$, hence $D_n - cL_n\left(0\right)$ is a martingale with respect to $\mathcal F_n$. 
    Further computation and the tower property yields
    \[
    \mathbb E\left[D_n -cL_n\left(0\right)\right]=0,\qquad
    \mathbb E\left[\left(D_n-cL_n\left(0\right)\right)^2\right]=c\left(1-c\right)\mathbb E\left[L_n\left(0\right)\right].
    \]
    Since $\mathbb E\left[L_n\left(0\right)\right]=\Theta\left(\sqrt n\right)$ for a recurrent aperiodic mean-zero finite-variance walk on $\mathbb Z$, we have
    \(
    \mathbb E\big[\left(\left(D_n-cL_n\left(0\right)\right)/\sqrt n\right)^2\big] = c\left(1-c\right)\mathbb E\left[L_n\left(0\right)\right]/n = O\left(n^{-1/2}\right)\to 0,
    \)
    which implies $\left(D_n-cL_n\left(0\right)\right)/\sqrt n \to 0$ in $L^2$ and hence in probability. Similarly, $\left(E_n-cL_n\left(0\right)\right)/\sqrt n \overset{P}{\to} 0$.
    
    The invariance principle for local times we use can be found in e.g. \cite{borodin1982asymptotic}. Specifically, by \cite[Theorem~1.1]{borodin1982asymptotic}, for a recurrent aperiodic mean-zero random walk with finite variance $\sigma^2$, the rescaled local time satisfies
    \[
    \frac{\sigma L_n\left(0\right)}{\sqrt{n}} \Rightarrow\ell_1\left(0\right),
    \]
    where $\ell_1\left(0\right)$ denotes the Brownian local time at~$0$ up to time~$1$.

    By Lévy’s identity, $\ell_1\left(0\right)$ has the same distribution as $|\mathcal{N}\left(0,1\right)|$, a folded normal random variable. We obtain
    \[
    \lim_{n \to \infty} \frac{\# \mathcal{C}_0^n} {\sqrt{n}} = \lim_{n \to \infty} \frac{D_n + E_n} {\sqrt{n}} \overset{d}{=} \lim_{n \to \infty} \frac{2x\left(1-x\right) L_0\left(n\right)}{\sqrt{n}} \overset{d}{=} |\mathcal{N}\left(0,2x\left(1-x\right)\right)|.
    \]
\end{proof}
\begin{corollary}
    Fix speed $U_0 = 2x -1$. The expected size of the convoy in the TASEP speed process, conditioned on the initial speed $U_0$, satisfies
    \begin{equation}
        \lim_{n \to \infty}\frac{\mathbb{E}_x\left[\# \mathcal{C}_0^n\right]}{\sqrt{n}} = \sqrt{\frac{4x\left(1-x\right)}{\pi}}.\label{fix}
    \end{equation}
    In particular, the unconditional expectation is given by
    \begin{equation}
        \lim_{n \to \infty}\frac{\mathbb{E}\left[\# \mathcal{C}_0^n\right]}{\sqrt{n}} = \frac{\sqrt{\pi}}{4}.\label{nonfix}
    \end{equation}
\end{corollary}
\begin{proof}
    The expectation of a folded Gaussian $|\mathcal{N}\left(0,\sigma\right)|$ is $\sigma \sqrt{\frac{2}{\pi}}$. Using \eqref{tasepconvoy}, we obtain
    \[
    \lim_{n \to \infty}\frac{\mathbb{E}_x\left[\# \mathcal{C}_0^n\right]}{\sqrt{n}} = \mathbb{E}\left[|\mathcal{N}\left(0,2x\left(1-x\right)\right)|\right] =  \sqrt{2x\left(1-x\right)}\sqrt{\frac{2}{\pi}} = \sqrt{\frac{4x\left(1-x\right)}{\pi}}.
    \]
    For non fixed speed $U_0$, note that $U_0$ is uniformly distributed on $\left[-1,1\right]$. Consequently, $x = \frac{1+U_0}{2}$ is uniformly distributed on $ \left[0,1\right]$. Hence, 
    \[
    \lim_{n \to \infty}\frac{\mathbb{E}\left[\# \mathcal{C}_0^n\right]}{\sqrt{n}} = \lim_{n \to \infty}\frac{\mathbb{E}\left[\mathbb{E}_x\left[\# \mathcal{C}_0^n\right]\right]}{\sqrt{n}} = \int_{0}^{1}\sqrt{\frac{4x\left(1-x\right)}{\pi}} dx = \frac{\sqrt{\pi}}{4}.
    \]
\end{proof}
\begin{remark}
    In the following sections, we restrict our proofs to the case of fixed $x$. To obtain \eqref{nonfix} as stated in the introduction, one can start from \eqref{fix} and proceed in exactly the same way: the only additional step is to integrate with respect to $x$, which is uniformly distributed on $\left[0,1\right]$. For simplicity, we omit this integration in the subsequent proofs. 
\end{remark}
\begin{remark}
    We cannot find an analogue proof for general $q$. The main difficulty is with parameter $q$ involved, $Q_i$ can produce a particle in the convoy at any level line with different probabilities. However, for $q = 0$ case, only the local time at level $0$ matters.
\end{remark}
\subsection{An exact formula for the ASEP convoy}
In this section, we formulate a exact formula for the expectation of the size of the ASEP convoy for fixed $q\in [0,1)$. In order to formally state the main theorem, we introduce q-Gandhi polynomials. 
\begin{definition}[\cite{han1999q}, p.2 (3) ]\label{Genocchi}
    Denote $\Delta_q$ as the q-Hahn operator, which is a q-analogue of difference operator on the polynomials.
    \[
    \Delta_q f\left(x\right) := \frac{f\left(1 + qx\right) - f\left(x\right)}{\left(1+qx\right) - x}.
    \]
    The q-Gandhi polynomials $B_n\left(x,q\right)$ are defined by the following recurrence, with $B_1\left(x,q\right) = 1$, and for $n \geq 2$,
    \[
    B_n\left(x,q\right) := \Delta_q\left(x^2  B_{n-1}\left(x,q\right)\right).
    \]
\end{definition}

\begin{remark}
    These q-Gandhi polynomials can be considered as a generalization of the well-known Gandhi polynomials by plugging in $q = 1$. On the other hand, if we plug in $x = 1$ in $B_n\left(x,q\right)$, the coefficients of the polynomials with respect to $q$ form the so called q-Genocchi numbers (A193762 in OEIS)
\end{remark}
\begin{example}We list the first several q-Gandhi polynomials and q-Genocchi numbers here as examples:
    
    The q-Gandhi polynomials:
    \begin{itemize}
        \item[] $B_1\left(x,q\right) = 1$, $B_2\left(x,q\right) = 1 + \left(1+q\right)x$,
        \item[] $B_3\left(x,q\right) = \left(2+q\right) + \left(2+ 4q + 2q^2\right)x + \left(1 + 2q + 2q^2 + q^3\right)x^2$,
    \end{itemize}
    
    The q-Genocchi numbers:
    \begin{itemize}
        \item[] $B_1\left(1,q\right) = 1$, $B_2\left(1,q\right) = 2+q$, $B_3\left(1,q\right) = 5 + 7q + 4q^2 + q^3$,
        \item[] $B_4\left(1,q\right) = 14 + 36q + 45q^2 + 35q^3 + 18q^4 + 6q^5 + q^6$.
    \end{itemize}
\end{example}

In this same paper \cite{han1999q}, the authors also compute the generating function of q-Genocchi numbers. 
\begin{proposition}[\cite{han1999q}, Corollary 5]\label{prop.genocchi}
    We have
    \[
    \sum_{n \geq 1} B_n\left(1,q\right) t^n = \sum_{n \geq 1}\frac{\left([n]_q!\right)^2 q^n t^n}{\left(q + \left[1\right]_q^2t\right)\left(q^2 + \left[2\right]_q^2t\right)\dots\left(q^n + \left[n\right]_q^2t\right)},
    \]
    where $\left[n\right]_q := 1 + q + \dots q^{n-1}$, $\left[n\right]_q! := \prod_{i=1}^{n} \left[i\right]_q $.  
\end{proposition}
In addition, by comparing the coefficient of $t^n$ on the both side of the generating function, we get an explicit formula for the q-Genocchi numbers.
\begin{proposition}
    For any $n\geq1$, we have
    \begin{equation}
        B_n\left(1,q\right)=\sum_{k=0}^{n-1}\left(-1\right)^{n-1-k}h_{n-1-k}\Big(\tfrac{\left[1\right]_q^2}{q},\tfrac{\left[2\right]_q^2}{q^2},\dots,\tfrac{\left[k+1\right]_q^2}{q^{k+1}}\Big)\big(\left[k{+}1\right]_q!\big)^{2}\prod_{j=1}^{k}q^{-j},\label{alter}
    \end{equation}
    where $h_k\left(x_1, \dots x_n\right)$ is the homogeneous symmetric function with degree $k$ and variable $x_1 \dots x_n$.
\end{proposition}
\begin{proof}
    By factoring out$q^i$ of each denominator term, we get
    \[
    \sum_{n\ge 1} B_n\left(1,q\right)t^n=\sum_{n\ge 1}\big(\left[n\right]_q!\big)^2q^{n-\binom{n+1}{2}}t^n\prod_{i=1}^{n}\Big(1+\tfrac{\left[i\right]_q^2}{q^i}t\Big)^{-1}.
    \]
    By the generating function for homogeneous symmetric functions
    \[
    \prod_{i=1}^{r}\frac{1}{1+x_i z} =\sum_{m\ge 0}\left(-1\right)^m h_m\left(x_1,\dots,x_r\right)z^m,
    \]
    we have
    \[
    \prod_{i=1}^{n}\Big(1+\tfrac{\left[i\right]_q^2}{q^i}t\Big)^{-1}=\sum_{m\ge0}\left(-1\right)^m h_m\Big(\tfrac{\left[1\right]_q^2}{q},\tfrac{\left[2\right]_q^2}{q^2},\dots,\tfrac{\left[n\right]_q^2}{q^n}\Big)t^m.
    \]
    Hence
    \[
    \sum_{n\ge 1} B_n\left(1,q\right)t^n =\sum_{n\ge 1} \big(\left[n\right]_q!\big)^2 q^{n-\binom{n+1}{2}}t^n \sum_{m\ge0}\left(-1\right)^m h_m\Big(\tfrac{\left[1\right]_q^2}{q},\dots,\tfrac{\left[n\right]_q^2}{q^n}\Big)t^m.
    \]
    Extracting the coefficient of $t^n$ gives the desired result. 
\end{proof}
\begin{theorem}\label{thm.main}
    The expectation of the difference of $Q_{n+1}$ and $Q_1$ can be written as a linear combination of q-Genocchi numbers, we have
    \begin{equation}
        \mathbb{E}_{x}\left[Q_{n+1} - Q_1\right] = \sum_{k=1}^{n}  \left(-1\right)^{k-1}\left(1-q\right)^{2k-1} c^{k}  \binom{n}{k} B_k\left(1,q\right), \label{exact}
    \end{equation}
    where $c=x\left(1-x\right)$.
\end{theorem}
\begin{remark}
    This formula was intended as a solution to the universality problem. However, the asymptotic analysis presents difficulties, since the generating function of the relevant combinatorial object in the exact formula has infinitely many poles near $0$. In the special case $q = 0$, the formula simplifies considerably, allowing us to recover the desired result. It would be interesting to develop a method for proving universality by applying asymptotic analysis directly to this general formula.
\end{remark}
\begin{proof}
    See section \ref{exactformula}.
\end{proof}

\section{Asymptotics size of the ASEP convoy}\label{sec.universal}
In this section, we show that the asymptotic size of the ASEP convoy is universal for fixed $q \in [0,1)$ , using a framework based on orthogonal polynomials. Although our approach shares technical similarities with \cite{bryc2025limits}, the present setting involves a different family of orthogonal polynomials, whose properties require a separate analysis. Understanding why closely related tools arise in these two distinct problems would be an interesting direction for future work. 

\subsection{The continuous big q-Hermite orthogonal polynomials}
There are several equivalent way to define a family of orthogonal polynomials, we define it via the recurrence relations. 
\begin{definition}[\cite{koekoek2010hypergeometric}, p.510 (14.18.4)]
For $|q|<1$ and $a\in[0,1)$, the continuous big $q$-Hermite polynomials $\{H_n\left(x;a\mid q\right)\}_{n=0}^\infty$ are defined by the recurrence relation
\begin{equation}\label{qhermite}
    2x H_n\left(x;a\mid q\right) 
    = H_{n+1}\left(x;a\mid q\right) + a q^n H_n\left(x;a\mid q\right) + \left(1-q^n\right) H_{n-1}\left(x;a\mid q\right),
    \qquad n \geq 1,
\end{equation}
with initial conditions
\[
H_0\left(x;a\mid q\right) = 1, 
\qquad H_{-1}\left(x;a\mid q\right) = 0.
\]
\end{definition}

\begin{remark}
    $H_n\left(x|q\right)$ is in a larger family of orthogonal polynomials which is called Askey-wilson polynomials. Askey-wilson polynomials is a four parameter orthogonal polynomials $H_n\left(x;a,b,c,d|q\right)$, defined firstly in \cite{askey1985some} and studied extensively later (see \cite{koekoek2010hypergeometric}). Take $\left(a,b,c,d\right) = \left(1,0,0,0\right)$ one get $H_n\left(x|q\right)$ from Askey-wilson polynomials. 
    
    We point out that this choice of parameters is not explicitly covered in the standard reference, since they only discuss the case when all parameters have module less than $1$, or one of them has module greater than $1$. Since the reference measure might be ill-defined when $|a| = 1$ in general settings. However, our choice of parameter can be justified by taking $a < 1$ and then let $a \to 1^{-}$. The recurrence definition $H_n\left(x;a|q\right)$ is stable under the limit $a \to 1^{-}$. Moreover, the density function $w\left(x\right)$ is continuous and strictly positive on $\left(-1,1\right)$, vanishes like $\sqrt{1+x}$ at $x = -1$ and has integrable singularity at $x = 1$. By dominated convergence, the orthogonality relation therefore extends to $a = 1$. Later in this paper, we only need the case where $a = 1$, we write $H_n\left(x;a|q\right)$ as $H_n\left(x|q\right)$ in short.
\end{remark}
We summary some properties of the $a = 1$ continuous big q-Hermite orthogonal polynomials, which are used later in this section. Most of the contents are from \cite{koekoek2010hypergeometric} by taking a $a\to 1^-$ limit. The next proposition shows that $H_n\left(x|q\right)$ has an unique reference measure. 
\begin{proposition}[\cite{koekoek2010hypergeometric}, p.510 (14.18.2)]\label{uniqueness}
    Let $|q|<1$ and $x=\cos\theta \in \left(-1,1\right)$.  
    For $\alpha \in \mathbb{C}$, define the weight function
    \begin{equation}
    w\left(x\right) := \frac{\left(q;q\right)_{\infty}}{2\pi \sqrt{1-x^{2}}} \frac{\left|\left(e^{2i\theta};q\right)_{\infty}\right|^{2}}{\left|\left(e^{i\theta};q\right)_{\infty}\right|^{2}},\quad x=\cos\theta.\label{refdef1}
\end{equation}

    $H_n\left(x|q\right)$ is orthogonal with respect to the reference measure $w\left(x\right)$, i.e.
    \begin{equation}
        \int_{-1}^{1} H_n\left(x|q\right) H_m\left(x|q\right) w\left(x\right) dx = \delta_{m,n} \left(q;q\right)_{n},\label{ortho}
    \end{equation}
    where $\delta_{m,n} = 1 \iff m=n$.
    In addition, $w\left(x\right)$ is the unique reference measure such that \eqref{ortho} holds. 
\end{proposition}
\begin{remark}
    Alternatively, $w\left(x\right)$ can be written in a different form, we have
    \begin{equation}
        w\left(x\right) = \frac{\left(q;q\right)_\infty}{2\pi} \frac{1}{\sqrt{1-x^2}} h\left(x,-1\right)h\left(x,\sqrt{q}\right)h\left(x,-\sqrt{q}\right), \qquad x\in\left(-1,1\right),\label{refdef2}
    \end{equation}
    where
    \[
    h\left(x,\alpha\right) := \prod_{k=0}^\infty \bigl(1-2\alpha xq^k + \alpha^2 q^{2k}\bigr)= \left(\alpha e^{i\theta},\alpha e^{-i\theta};q\right)_\infty.
    \]
\end{remark}

\begin{proposition}[\cite{koekoek2010hypergeometric}, p.512 (14.18.13)]
    The generating function of $H_n\left(x|q\right)$ is the following.
    \begin{equation}
    \sum_{n=0}^{\infty} \frac{H_{n}\left(\cos\theta \mid q\right)}{\left(q;q\right)_{n}} t^{n} = \frac{\left(t; q\right)_{\infty}}{\left(e^{i\theta}t, e^{-i\theta}t; q\right)_{\infty}}.\label{Hgenerating}
    \end{equation}
\end{proposition}
\begin{remark}
    To match the notation of a theorem that we use later, we define a normalized version of $H_{n}\left(x|q\right)$. We denote it by $\hat{H}_{n}\left(x|q\right)$. Let 
    $\hat{H}_{n}\left(x|q\right) = \frac{H_{n}\left(x|q\right)}{\sqrt{\left(q;q\right)_{\infty}}}$, $\hat{H}_{n}\left(x|q\right)$, which satisfies the recurrence relation:
    \[
    x \hat{H}_{n}\left(x|q\right) = \frac{\sqrt{1 - q^{n+1}}}{2} \hat{H}_{n+1}\left(x|q\right) + \frac{q^n}{2} \hat{H}_{n}\left(x|q\right) + \frac{\sqrt{1 - q^{n}}}{2} \hat{H}_{n-1}\left(x|q\right).
    \]
    With respect to the same reference measure, we have
    \[
    \int_{-1}^{1} \hat{H}_n\left(x|q\right) \hat{H}_m\left(x|q\right) w\left(x\right) dx = \delta_{m,n}.
    \]
\end{remark}

We require pointwise asymptotics for the $H_n\left(x|q\right)$ at the upper endpoint of
the orthogonality interval. Such pointwise limits have been studied, see for example \cite{aptekarev1993asymptotics, 10.1093/imrn/rny042}). For our purposes, it suffices to invoke a result established in \cite{aptekarev1993asymptotics}.

\begin{proposition}[\cite{aptekarev1993asymptotics}, Theorem 1]\label{prop.near}
    Consider sequences $\{q_n\left(x\right)\}_{n = 0}^{\infty}$ of orthogonal polynomials with respect to $d\mu\left(x\right)$, such that
    \begin{equation}
        b_n q_{n+1}\left(x\right) + a_n q_{n}\left(x\right) + b_{n-1} q_{n-1}\left(x\right) = xq_n\left(x\right). \label{near}
    \end{equation}
    Suppose that $\lim_{n \to \infty } b_n = \frac{1}{2}$, $\lim_{n \to \infty } a_n = 0$, in addition,
    \begin{equation}
        \frac{q_{n+1}\left(1\right)}{q_n\left(1\right)}  \simeq 1 + \frac{\alpha+\frac{1}{2}}{n} + o\left(1\right), \quad \alpha > -1.\label{near2}
    \end{equation}
    Let $\hat{q}_n\left(x\right) = \frac{q_n\left(x\right)}{q_n\left(1\right)}$, 
    \begin{equation}
        \lim_{n \to \infty}\hat{q}_n\left(1 - \frac{z^2}{2 n^2}\right)  = 2^\alpha \Gamma\left(\alpha+1\right)z^{-\alpha}J_{\alpha}\left(z\right),
    \end{equation}
    where $J_{\alpha}\left(z\right)$ is the first kind Bessel function. The convergence holds uniformly for $z$ in a compact subset of $\mathbb{R}$.
\end{proposition}
As a corollary, we have the pointwise asymptotics for the $H_n\left(x|q\right)$ near $1$.
\begin{corollary}
    Fix $q \in [0,1)$, let $H_n\left(x|q\right)$ be the $a = 1$ continuous big q-Hermite polynomial, then for $u\in \left(0,1\right)$,
    \begin{equation}
        \lim_{n \to \infty} H_n\left(1 - \frac{u^2}{2 n^2}\right) = \cos{u}. \label{polynear}
    \end{equation}
\end{corollary}
\begin{proof}
    The recurrence relation of $\hat{H}_n\left(x|q\right)$ satisfies condition \eqref{near} in Proposition \ref{prop.near}, with $b_n = \frac{\sqrt{1 - q^{n+1}}}{2} $ and $a_n = \frac{q^n}{2}$. Clearly, $\lim_{n \to \infty } b_n = \frac{1}{2}$, $\lim_{n \to \infty } a_n = 0$ is satisfied. By definition $H_{n}\left(1\right)  = 1$, hence $\hat{H}_n\left(1\right) = \frac{H_{n}\left(1\right)}{\sqrt{\left(q;q\right)_n}} = \frac{1}{\sqrt{\left(q;q\right)_n}}$, condition \eqref{near2} holds with $\alpha = -\frac{1}{2}$. By applying \eqref{prop.near} to $\hat{H}_n$, and recalling that $H_n\left(x\right)=\frac{\hat{H}_n\left(x\right)}{\hat{H}_n\left(1\right)}$, together with the identity for the Bessel function of the first kind $J_{-\frac{1}{2}}\left(u\right) = \sqrt{\frac{2}{\pi u}} \cos{u}$, we obtain
    \[
    \lim_{n \to \infty} H_n\left(1 - \frac{u^2}{2 n^2}\right) = 2^{-\frac{1}{2}} \Gamma\left(\frac{1}{2}\right) \sqrt{\frac{2}{\pi u}} \cos{u} = \cos{u}.
    \]
\end{proof}

Recall that the density function $w\left(x\right)$ has integrable singularity at $x = 1$. We present here a more precise description for later usage. For this purpose, we need the following property of q-Pochhammer symbol.
\begin{lemma}
    For $q \in [0,1)$,
    \begin{equation}
        \left(-q;q\right)_{\infty}\left(q;q^2\right)_{\infty} = 1.\label{qiden}
    \end{equation}
\end{lemma}

\begin{proof}
    \[
    \begin{aligned}
         \left(-q;q\right)_{\infty}\left(q;q^2\right)_{\infty}\left(q;q\right)_{\infty}&= \prod_{k=1}^{\infty}\left(1 + q^k\right) \prod_{k=1}^{\infty}\left(1 - q^k\right)\prod_{k=0}^{\infty}\left(1 - q^{2k+1}\right) \\
        &=  \prod_{k=1}^{\infty}\left(1 - q^{2k}\right)\prod_{k=0}^{\infty}\left(1 - q^{2k+1}\right)= \prod_{k=1}^{\infty}\left(1 - q^k\right)= \left(q;q\right)_{\infty}.
    \end{aligned}
    \]
    Clearly, $\left(q;q\right)_{\infty} \neq 0$, after canceling out $\left(q;q\right)_{\infty}$ from both side, we derive $\left(-q;q\right)_{\infty}\left(q;q^2\right)_{\infty} = 1$.
\end{proof}

\begin{lemma}
    Let $w\left(x\right)$ be the density function of the reference measure of $H_n\left(x|q\right)$, supported on $\left[-1,1\right]$. For fixed $q \in [0,1)$, we have 
    \begin{equation}
        \lim_{\epsilon \to 0} \sqrt{\epsilon}w\left( 1 - \epsilon\right) = \frac{\sqrt{2} \left(q;q\right)_{\infty}}{\pi}. \label{dennear}
    \end{equation}
    In addition, $\sqrt{1-x}w\left( x\right)$ has a uniform upper bound for $x \in [0,1]$.
\end{lemma}
\begin{proof}
    By the definition of $w\left(x\right)$ \eqref{refdef2}, we have
    \[
    \begin{aligned}
        w\left(x\right) &= \frac{\left(q;q\right)_{\infty}}{2 \pi} \frac{1}{\sqrt{1 - x^2}} h\left(x,-1\right) h\left(x,\sqrt{q}\right) h\left(x,-\sqrt{q}\right)\\
        &=  \frac{\left(q;q\right)_{\infty}}{2 \pi} \frac{1}{\sqrt{1 - x^2}}  \prod_{k=0}^{\infty}\left(1 + 2 x q^k + q^{2k}\right)\prod_{k=0}^{\infty}\left(1 - 2 x q^{k+\frac{1}{2}} + q^{2k+1}\right) \prod_{k=0}^{\infty}\left(1 + 2 x q^{k+\frac{1}{2}} + q^{2k+1}\right)\\
        &= \frac{\left(q;q\right)_{\infty}}{2 \pi} \frac{2+2x}{\sqrt{1 - x^2}}  \prod_{k=1}^{\infty}\left(1 + 2 x q^k + q^{2k}\right)\prod_{k=0}^{\infty}\left(1  + 2q^{2k+1} + q^{4k+2} - 4 x^2 q^{2k+1}\right).
    \end{aligned}
    \]
    Let $r\left(x\right) := \prod_{k=1}^{\infty}\left(1 + 2 x q^k + q^{2k}\right)\prod_{k=0}^{\infty}\left(1  + 2q^{2k+1} + q^{4k+2} - 4 x^2 q^{2k+1}\right)$.
    \[
    \begin{aligned}
        \lim_{x \to 1} r\left(x\right) &= \lim_{x \to 1}\prod_{k=1}^{\infty}\left(1 + 2 x q^k + q^{2k}\right)\prod_{k=0}^{\infty}\left(1  + 2q^{2k+1} + q^{4k+2} - 4 x^2 q^{2k+1}\right)\\
        &= \prod_{k=1}^{\infty}\left(1 + 2 q^k + q^{2k}\right)\prod_{k=0}^{\infty}\left(1  + 2q^{2k+1} + q^{4k+2} - 4 q^{2k+1}\right) = \left(-q;q\right)_{\infty}^{2}\left(q;q^2\right)_{\infty}^{2}.\\
    \end{aligned}
    \]
    We obtain $\lim_{x \to 1} r\left(x\right) = 1$ by using \eqref{qiden}. Finally, 
    \[
    \lim_{x \to 1} \sqrt{1-x}  w\left(x\right) = \lim_{x \to 1}r\left(x\right) \lim_{x \to 1} \frac{\left(q;q\right)_{\infty}}{ \pi} \sqrt{1+x} = \frac{\sqrt{2} \left(q;q\right)_{\infty}}{\pi}.
    \]
    For the uniform bound of $\sqrt{1-x}w\left( x\right)$, we have
    \begin{equation}
        \sqrt{1-x}w\left( x\right) =  \frac{\left(q;q\right)_{\infty}}{ \pi} r\left(x\right) \sqrt{1+x}  \le \frac{\sqrt{2}}{\pi}\left(q;q\right)_{\infty}\left(-q;q\right)_{\infty}^{4} := C\left(q\right). \label{Uniboundw}
    \end{equation}
\end{proof}

The key difference between the polynomials \(H_n\) considered in this work and those studied in \cite{bryc2025limits} lies in their uniform boundedness. In \cite{bryc2025limits}, the orthogonal polynomials admit a simple uniform bound by their value at $1$ (see \cite[Corollary~3.1]{bryc2025limits}). By contrast, our polynomials \(H_n\) do not enjoy such a convenient property: at \(x=-1\), they exhibit large oscillations. Nevertheless, we establish a uniform bound for \(H_n\) on every compact subset of \((-1,1]\). This weaker but sufficient control still allows us to carry out the asymptotic analysis. The method we use to prove such a bound is called singularity analysis of generating functions, which is a powerful tool in analytical combinatorics \cite{flajolet2009analytic}. 

\begin{proposition}
    Fix $0<q<1$ and $\theta_{0}\in\left(0,\pi\right)$. Then there exists a constant $C=C\left(q,\theta_{0}\right)>0$ such that for all $n\geq 0$ and all $\theta\in[0,\theta_{0}]$ (equivalently, $x=\cos\theta\in[\cos\theta_{0},1]$) one has
    \begin{equation}
        |H_{n}\left(\cos\theta\mid q\right)| \leq C.\label{Hunibound}
    \end{equation}
\end{proposition}

\begin{proof}
The generating function of $H_n$ is \eqref{Hgenerating}
\[
\sum_{n=0}^{\infty}\frac{H_{n}\left(\cos\theta\mid q\right)}{\left(q;q\right)_{n}}t^{n}
=\frac{\left(t;q\right)_{\infty}}{\left(e^{i\theta}t;q\right)_{\infty}\left(e^{-i\theta}t;q\right)_{\infty}}
=:F_{\theta}\left(t\right),\qquad |t|<1.
\]
On the unit circle $|t|=1$, $F_{\theta}$ has two simple poles at $t=e^{\pm i\theta}$. Subtracting their principal parts gives the decomposition
\[
F_{\theta}\left(t\right)
= \frac{A_{+}}{1-e^{i\theta}t}
+ \frac{A_{-}}{1-e^{-i\theta}t}
+ R_{\theta}\left(t\right),
\]
where $R_{\theta}$ is analytic on $|t|\leq 1$. A residue computation yields
\[
A_{+}=\frac{\left(e^{-i\theta};q\right)_{\infty}}{\left(q;q\right)_{\infty}\left(e^{-i2\theta};q\right)_{\infty}},
\qquad
A_{-}=\frac{\left(e^{i\theta};q\right)_{\infty}}{\left(q;q\right)_{\infty}\left(e^{i2\theta};q\right)_{\infty}}.
\]
Expanding the geometric series
\(
\left(1-e^{\pm i\theta}t\right)^{-1}=\sum_{n\ge0}e^{\pm in\theta}t^{n}
\)
and comparing coefficients shows that
\[
\frac{H_{n}\left(\cos\theta\mid q\right)}{\left(q;q\right)_{n}}
=A_{+}e^{in\theta}+A_{-}e^{-in\theta}+r_{n}\left(\theta\right),
\]
where $r_{n}\left(\theta\right)=\left[t^{n}\right]R_{\theta}\left(t\right)$. Since $R_{\theta}$ is analytic on $|t|\leq 1$, Cauchy’s formula implies
\[
|r_{n}\left(\theta\right)|\le \sup_{|t|=1}|R_{\theta}\left(t\right)|.
\]
This quantity depends continuously on $\theta$ and is therefore bounded on the compact interval $\left[0,\theta_{0}\right]$ by $M\left(q,\theta_{0}\right)$.

Using the elementary bounds valid for $|z|=1$,
\[
\left(q;q\right)_{\infty}|1-z|\ \le\ |\left(z;q\right)_{\infty}|\ \le\ |1-z|\left(-q;q\right)_{\infty},
\]
and applying them to $z=e^{-i\theta}$ and $z=e^{-i2\theta}$, we obtain
\[
|\left(e^{-i\theta};q\right)_{\infty}|\le |1-e^{-i\theta}|\left(-q;q\right)_{\infty},\qquad
|\left(e^{-i2\theta};q\right)_{\infty}|\ge |1-e^{-i2\theta}|\left(q;q\right)_{\infty}.
\]
Hence
\[
|A_{+}|\le
\frac{|1-e^{-i\theta}|}{|1-e^{-i2\theta}|}
\frac{\left(-q;q\right)_{\infty}}{\left(q;q\right)_{\infty}^{2}}
=\frac{1}{2\cos\left(\theta/2\right)}\frac{\left(-q;q\right)_{\infty}}{\left(q;q\right)_{\infty}^{2}}.
\]
The same estimate holds for $A_{-}$. Therefore, for $\theta\in[0,\theta_{0}]$,
\begin{equation}\label{eq:AmpBound}
|A_{\pm}|\ \le\ \frac{1}{2\cos\left(\theta_{0}/2\right)}\frac{\left(-q;q\right)_{\infty}}{\left(q;q\right)_{\infty}^{2}}
=:C_{A}\left(q,\theta_{0}\right).
\end{equation}

Since $\left(q;q\right)_{n} < 1$, we obtain
\[
|H_{n}\left(\cos\theta\mid q\right)| \le 2C_{A}\left(q,\theta_{0}\right)+M\left(q,\theta_{0}\right)=:C\left(q,\theta_{0}\right),
\]
uniformly in $n\ge0$ and $\theta\in[0,\theta_{0}]$.
\end{proof}
\begin{remark}
    From the proof it also becomes clear why a uniform bound for \(H_n\left(x\right)\) for all \(x \in [-1,1]\) and all \(n\) cannot exist. At the endpoint \(x = -1\) (equivalently, \(\theta = \pi\)), the generating function \(F_{\theta}\) has a double pole at \(t = -1\), which is its dominant singularity. As a consequence, \(H_n\left(-1\right)\) grows linearly with \(n\).
\end{remark}

\subsection{Birth and death chains and orthogonal polynomials}
The motivation for introducing orthogonal polynomials lies in their connection with random walks, particularly in the context of birth–death processes. Recall that the stochastic process $Q_n$ which characterises the properties of the convoy is a birth–death chain. The following result is known in the literature as the Karlin–McGregor theory \cite{karlin1959random}.
\begin{definition}
    Let $X_n$ be a discrete Markov chain on $\mathbb{N}$, with transition probabilities
    \[
    \mathbb{P}_{ij} := \mathbb{P}\left(X_{n+1} = j | X_n = i\right).
    \]
    We say that $X_n$ is a birth-death chain if its transition probabilities $\mathbb{P}_{ij}$ is tridiagonal; that is, 
    \[
    P_{ij} = 0, \text{if} \quad |i-j|>1.
    \]
    In this case, the nonzero entries are given by $P_{i,i} = r_i, P_{i,i+1} = p_i, P_{i,i-1} = q_i$. We also define the m step transition probabilities,
    \[
    \mathbb{P}^{m}_{ij} := \mathbb{P}\left(X_{n+m} = j | X_n = i\right).
    \]
\end{definition}
\begin{proposition}[\cite{karlin1959random}, Theorem 1]
    Let $\left(X_n\right)_{n \geq 0}$ be a birth--death chain with transition probabilities $\left(p_i, q_i, r_i\right)$, and define $\theta_n = \frac{\prod_{i=0}^{n-1}p_i}{\prod_{i=1}^{n}q_i}$. Consider the sequence of polynomials $\{R_n\}_{n \geq 0}$ determined by the three-term recurrence
    \begin{equation}
        \left\{
        \begin{aligned}
            &xR_n\left(x\right)  = p_n R_{n+1}\left(x\right) + r_n R_{n}\left(x\right) + q_n R_{n-1}\left(x\right), \\
            &R_0\left(x\right) =1, R_{-1}\left(x\right) = 0.
        \end{aligned}
        \right.\label{randomwalk}
    \end{equation}
    Then there is an unique positive density function $\psi\left(x\right)$ such that  $\text{supp}\left(\psi\right) \subset [-1,1]$ and the $n$-step transition probabilities of $X_n$ satisfy
    \begin{equation}
        \mathbb{P}^{n}_{ij} = \theta_j \int_{-1}^{1} x^n R_i\left(x\right) R_{j}\left(x\right)\psi\left(x\right) dx.\label{transp}
    \end{equation}   
\end{proposition}

\begin{remark}
    Note that the polynomials $R_n$ are orthogonal with respect to $\psi\left(x\right)$. It follows from plugging $n=0$ into \eqref{transp}. 
\end{remark}
This proposition enables us to analyze the queuing process $Q_n$ via orthogonal polynomials, since $Q_n$ is a birth-death chain with parameters $p_i = c$, $q_i = c\left(1 - q^i\right)$ and $r_i = 1 - 2c + cq^i$.

\begin{proposition}
    The transition probabilities of the queuing process $Q_n$ \eqref{queue} can be expressed by orthogonal polynomials. Specifically, We have 
    \begin{equation}
        \mathbb{P}\left(Q_{n+1} = j| Q_1 = i\right) = \frac{1}{\left(q;q\right)_j}\int_{1 - 4c}^{1}x^n P_i\left(x\right) P_{j}\left(x\right)\psi\left(x\right) dx,\label{Qtrans}
    \end{equation}
    where $P_n\left(x\right) = H_n\left(\frac{x}{2c}-\frac{1}{2c}+1\right)$, $H_n\left(x\right)$ is the $a = 1 $ continuous big q-Hermite polynomials, $\psi\left(x\right) = \frac{1}{2c} w\left(\frac{x}{2c}-\frac{1}{2c}+1\right)$, $w\left(x\right)$ is the reference measure of $H_n\left(x\right)$. 
    Equivalently, 
    \begin{equation}
        \mathbb{P}\left(Q_{n+1} = j| Q_1 = i\right) = \frac{1}{\left(q;q\right)_j}\int_{-1}^{1}\left(1-2c+2cz\right)^n H_i\left(x\right) H_{j}\left(x\right)w\left(x\right) dx.\label{Htrans}
    \end{equation}
\end{proposition}

\begin{proof}
    It is obvious that $\theta_n = \frac{\prod_{i=0}^{n-1}p_i}{\prod_{i=1}^{n}q_i} = \frac{c^n}{c^n \prod_{i=1}^{n}\left(1 - q^i\right)} = \frac{1}{\left(q;q\right)_n}.$
    It is sufficient to prove that $P_n\left(x\right)$ and $H_n\left(x\right)$ are the same up to a change of variable, the statement of the density follows directly from the orthogonality and the uniqueness of reference measure provided by Proposition \ref{uniqueness}. 
    
    Using \eqref{randomwalk} with $p_n = c$, $q_n = c\left(1 - q^n\right)$, and $r_n = 1 - 2c + cq^n$, the orthogonal polynomials $P_n$ associated with the random walk $Q_n$ satisfy the recurrence relation
    \[
    xP_n\left(x\right)  = c P_{n+1}\left(x\right) + \left(1 - 2c + cq^n\right) P_{n}\left(x\right) + c\left(1 - q^n\right) P_{n-1}\left(x\right).
    \]
    To prove that $P_n\left(x\right) = H_n\left(\frac{x}{2c}-\frac{1}{2c}+1\right)$, it is sufficient to show that $H_n\left(\frac{x}{2c}-\frac{1}{2c}+1\right)$ satisfy the same recurrence relation. Let $y = \frac{x}{2c}-\frac{1}{2c}+1$,
    \[
    \begin{aligned}
        &\quad \quad \quad xH_n\left(y\right)   = c H_{n+1}\left(y\right) + \left(1 - 2c + cq^n\right) H_{n}\left(y\right) + c\left(1 - q^n\right) H_{n-1}\left(y\right)\\
        &\iff \left(2cy + 1 -2c\right)H_n\left(y\right)  = c H_{n+1}\left(y\right) + \left(1 - 2c + cq^n\right) H_{n}\left(y\right) + c\left(1 - q^n\right) H_{n-1}\left(y\right)\\
        &\iff 2y H_n\left(y\right)  = H_{n+1}\left(y\right) + q^n H_{n}\left(y\right) + \left(1 - q^n\right) H_{n-1}\left(y\right).
    \end{aligned}
    \]
    Using the recurrence relation \eqref{qhermite} for $H\left(x\right)$,
    \[
    2x H_n\left(x\right) = H_{n+1}\left(x\right) + q^n H_n\left(x\right) + \left(1 - q^n \right)H_{n-1}\left(x\right).
    \]
    Therefore, the polynomials $P_n\left(x\right)$ and $H_n\left(\frac{x}{2c}-\frac{1}{2c}+1\right)$ coincide. 
\end{proof}
\subsection{Proof of Theorem \ref{thm.3}}
Before we proceed to the proof of Theorem \ref{thm.3}, we formulate a technical lemma about an improper integral, which can be proved via elementary calculus. 
\begin{lemma}
    For any $a \in \mathbb{R}$, we have
    \begin{equation}
        \int_{0}^{\infty} e^{-x^2} \cos{ax} dx = \frac{\sqrt{\pi}}{2} e^{-\frac{1}{4}a^2}.\label{improper}
    \end{equation}
\end{lemma}
\begin{theorem}\label{thm.uni}
    The inhomogeneous random walk $Q_n$ converges in distribution to a q-independent limit after scaling. That is, for any $x\in\left(0,1\right)$, $c = x\left(1-x\right) \in \left(0, \frac{1}{4}\right)$, we have
    \begin{equation}
        \lim_{n \to \infty} \frac{Q_{n+1}}{\sqrt{n}} \overset{\left(d\right)}{=} |\mathcal{N}\left(0,2c\right)|.\label{scaling}
    \end{equation}
    Moreover, conditionally on the event $U_0 = u$ with $x = \tfrac{1+u}{2}$, the size of the convoy is asymptotically universal, 
    \[
    \lim_{n \to \infty} \frac{\mathbb{E}_x \left[\mathcal{C}_0^{n}\right]}{\sqrt{n}} =  \sqrt{\frac{4c}{\pi}}. 
    \]
\end{theorem}

\begin{remark}
    It is possible to pursue a stronger convergence result for $Q_{\lfloor ns\rfloor}$ as a stochastic process, as in \cite{bryc2025limits}. However, we do not prove this, as the scaling limit for random variable $Q_n$ already provides sufficient information about the size of the convoy. 
\end{remark}
\begin{lemma}\label{lem.local0}
    Under the condition of $Q_1 = 0$, the local limit theorem for $\mathbb{P}\left(Q_{n+1} = \lfloor y \sqrt{n} \rfloor\right)$ holds. We have
    \begin{equation}
        \lim_{n \to \infty} \sqrt{n}\mathbb{P}\left(Q_{n+1} = \lfloor y \sqrt{n} \rfloor | Q_1 = 0\right) = \frac{1}{\sqrt{c\pi}} e^{-\frac{y^2}{4c}}.
    \end{equation}
    Thus, 
    \[
    \lim_{n \to \infty} \frac{Q_{n+1}}{\sqrt{n}}|_{Q_1 = 0} \overset{\left(d\right)}{=} |\mathcal{N}\left(0,2c\right)|.
    \]
\end{lemma}
\begin{proof}
    We apply \eqref{Htrans} to $i = 0$, $j = \lfloor y \sqrt{n} \rfloor$, using the fact that $H_0 = 1$, 
    \[
    \sqrt{n}\mathbb{P}\left(Q_{n+1} = \lfloor y \sqrt{n} \rfloor| Q_1 = 0\right) = \frac{\sqrt{n}}{\left(q;q\right)_{\lfloor y \sqrt{n} \rfloor}}\int_{-1}^{1} \left(2cz + 1 - 2c\right)^n H_{\lfloor y \sqrt{n} \rfloor}\left(z\right) w\left(z\right) \; dz
    \]
    We observe that the integral  
    \[
    \int_{-1}^{1} \left(2cz + 1 - 2c\right)^n  H_{\lfloor y \sqrt{n} \rfloor}\left(z\right) w\left(z\right) \; dz
    \]
    is negligible on the interval \(\left[-1,0\right]\). To establish this, we first derive an upper bound for \(H_n\).

    \begin{lemma}
        Let $h_n:=\sup_{x\in[-1,1]}H_n\left(x\right)$ be the supremum of $H_n$ on $\left[-1,1\right]$. Then for any $n \geq 0$,
        \begin{equation}
            h_n \leq 4^n.\label{bound}
        \end{equation}
    \end{lemma}
    \begin{proof}
    $H_0 = 1 $, $H_1 = 2x - 1 \leq 3 < 4$, so the inequality holds for $n = 0,1$. For any $n\geq 2$, we prove it by induction.
    By definition \eqref{qhermite}, 
        \[
        H_{n+1}\left(x\right) = \left(2x-q^n\right) H_n\left(x\right) - \left(1-q^n\right) H_{n-1}\left(x\right).
        \]
    Taking the absolute value and the supremum on both sides of the inequality, we obtain
    \[
    h_{n+1} = \sup_{x\in[-1,1]} \big| \left(2x-q^n\right) H_n\left(x\right) - \left(1-q^n\right) H_{n-1}\left(x\right) \big| < 3 h_{n} + h_{n-1} < 4^{n+1}.
    \]
    \end{proof}
    When $z \in [-1,0]$ and $c \in \left(0, \frac{1}{4}\right)$, $2cz + 1 - 2c < 1-2c< 1$. We have 
    \[
    \big|\int_{-1}^{0} \left(2cz + 1 - 2c\right)^n H_{\lfloor y \sqrt{n} \rfloor}\left(z\right) w\left(z\right) \; dz\big|< \left(1-2c\right)^n h_{y\sqrt{n}} \int_{-1}^{0} w\left(z\right) \; dz < \left(1-2c\right)^n 4^{y\sqrt{n}}.
    \]
    We conclude by the fact that $\left(1-2c\right)4^{y\sqrt{n}} \to 0$, when $n \to \infty$. 
    From the previous observation, we obtain  
    \[
    \begin{aligned}
    \sqrt{n}\mathbb{P}\left(Q_{n+1} = \lfloor y \sqrt{n} \rfloor \middle| Q_1 = 0\right) 
    &= \frac{\sqrt{n}}{\left(q;q\right)_{\lfloor y \sqrt{n} \rfloor}} \int_{0}^{1} \left(2cz + 1 - 2c\right)^n  H_{\lfloor y \sqrt{n} \rfloor}\left(z\right) w\left(z\right) \; dz + o\left(1\right).
    \end{aligned}
    \]

    We now perform the change of variables
    \[
    z = 1 - \frac{u^2}{2n}, \qquad \; dz = -\frac{u}{n}du,
    \]
    which maps \(z \in [0,1]\) to \(u \in [\sqrt{2n},0]\). Substituting into the integral gives
    \begin{equation}
        \sqrt{n}\mathbb{P}\left(Q_{n+1} = \lfloor y \sqrt{n} \rfloor \middle| Q_1 = 0\right) = \frac{1}{\left(q;q\right)_{\lfloor y \sqrt{n} \rfloor}}\int_{0}^{\sqrt{2n}}\left(1 - \frac{cu^2}{n}\right)^{n} H_{\lfloor y \sqrt{n} \rfloor}\left(1 - \tfrac{u^2}{2n}\right) w\left(1 - \tfrac{u^2}{2n}\right)\frac{udu}{\sqrt{n}}.\label{integrand}
    \end{equation}
    Letting $n \to \infty$, we have the following limits for any $u \in \mathbb{R}$ and $y >0$,
    \[
    \lim_{n \to \infty}\left(1 - \frac{cu^2}{n}\right)^{n} = e^{-cu^2}, \lim_{n \to \infty}\frac{1}{\left(q;q\right)_{\lfloor y \sqrt{n} \rfloor}} = \frac{1}{\left(q;q\right)_{\infty}}.
    \]
    To analyze the asymptotics of the integrand, we apply \eqref{polynear} and \eqref{dennear} with the appropriate change of variables. 
    By \eqref{polynear} with $u \mapsto yu$ and $n \mapsto y\sqrt{n}$, we obtain
    \[
    \lim_{n \to \infty} H_{\lfloor y \sqrt{n} \rfloor}\left(1 - \frac{u^2}{2n}\right) = \cos\left(yu\right).
    \]
    Similarly, setting $\epsilon = \tfrac{u^2}{2n}$ in \eqref{dennear} gives
    \[
    \lim_{n \to \infty} \frac{u}{\sqrt{n}} w\left(1 - \frac{u^2}{2n}\right) = \frac{2}{\pi} \left(q;q\right)_{\infty}.
    \]
    By using the bounds \eqref{Hunibound} and \eqref{Uniboundw} that we proved in the previous section, we have the following bounds (uniform in $n$) for the integrand in \eqref{integrand} before the limit:
    \[
    \left(1 - \frac{cu^2}{n}\right)^{n} \le e^{-cu^2}, \qquad H_{\lfloor y\sqrt{n} \rfloor}\left(1 - \frac{u^2}{2n}\right) \le C\left(q, \theta_0 = \frac{\pi}{2}\right) \qquad \frac{u}{\sqrt{n}} w\left(1 - \frac{u^2}{2n}\right) \le C\left(q\right),
    \]
    Hence, by applying the dominated convergence theorem, we obtain,
    \[
    \lim_{n \to \infty}\sqrt{n}\mathbb{P}\left(Q_{n+1} = \lfloor y \sqrt{n} \rfloor| Q_1 = 0\right) = \frac{2}{\pi}\int_{0}^{\infty} e^{-cu^2} \cos{yu} \;du.
    \]
    Apply the change of variables $u=v/\sqrt{c}$. Then $du=dv/\sqrt{c}$ and
    \[
    \int_{0}^{\infty} e^{-c u^{2}}\cos\left(yu\right)du =\frac{1}{\sqrt{c}}\int_{0}^{\infty} e^{-v^{2}}\cos\left(\frac{y}{\sqrt{c}}v\right)dv.
    \]
    By lemma \eqref{improper} with $a=y/\sqrt{c}$, we obtain
    \[
    \int_{0}^{\infty} e^{-c u^{2}}\cos\left(yu\right)du =\frac{1}{\sqrt{c}}\cdot \frac{\sqrt{\pi}}{2}e^{-\frac{y^{2}}{4c}}.
    \]
    Multiplying by $\frac{2}{\pi}$ yields
    \[ \lim_{n\to\infty}\sqrt{n}\mathbb{P}\left(Q_{n+1}=\bigl\lfloor y\sqrt{n}\bigr\rfloor \middle| Q_1=0\right) =\frac{1}{\sqrt{c\pi}}e^{-\frac{y^2}{4c}},
    \]
    as claimed in the first statement of Lemma \ref{lem.local0}.

    As for the second statement of the lemma, observe that 
    \[
    \frac{1}{\sqrt{c\pi}}e^{-\frac{y^2}{4c}}, \quad y \in [0,\infty),
    \]
    is the density of the random variable \(|\mathcal{N}\left(0,2c\right)|\). The passage from the local limit theorem to the corresponding weak convergence statement is standard in probability theory; see, for example, \cite[Theorem~3.3]{billingsley2013convergence}. \footnote{When the pointwise local limit theorem holds, one can consider $\frac{Q_{n+1} + U_n}{\sqrt{n}}$, where $U_n$ is a family of i.i.d uniform random variable in $[0,1]$. This new sequence of variables $\frac{Q_{n+1} + U_n}{\sqrt{n}}$ has density, which converges to the density $\frac{1}{\sqrt{c\pi}}e^{-\frac{y^2}{4c}}$. The convergence of density implies convergence in distribution. By Slutsky's theorem, $\frac{Q_{n+1}}{\sqrt{n}}$ also converges in distribution to the same limit.}
\end{proof}

\begin{proof}[Proof of the theorem]
    We first establish, for any fixed $k$, that  
    \begin{equation}
        \lim_{n \to \infty} \sqrt{n}\mathbb{P}\left(Q_{n+1} = \lfloor y \sqrt{n} \rfloor | Q_1 = k\right) = \frac{1}{\sqrt{c\pi}} e^{-\frac{y^2}{4c}}.\label{localk}
    \end{equation}
    
    To see this, we apply \eqref{Qtrans} to $\mathbb{P}\left(Q_{n+1} = \lfloor y \sqrt{n} \rfloor \mid Q_1 = k\right)$ with $i = k$ and $j = \lfloor y \sqrt{n} \rfloor$. The only difference compared to the case $Q_1 = 0$ is that while $P_0 = 1$, for general $k$, $P_k$ is a polynomial of degree $k$. However, from the definition of $P_k$ \eqref{randomwalk}, substituting $x = 1$ gives
    \[
    P_k\left(1\right) = p_nP_{k+1}\left(1\right) + q_nP_k\left(1\right) + r_nP_{k-1}\left(1\right).
    \]
    We therefore deduce that $P_k\left(1\right) = 1$ for all $k \geq 0$, which implies that the sum of the coefficients of $P_k$ equals $1$. It follows that 
    \[
    \frac{\sqrt{n}}{\left(q;q\right)_{\lfloor y \sqrt{n} \rfloor}}\int_{-1}^{1} x^n P_k\left(x\right) P_{\lfloor y \sqrt{n} \rfloor}\left(x\right) \psi\left(x\right)  dx
    \]
    can be expressed as a linear combination of terms of $\frac{\sqrt{n}}{\left(q;q\right)_{\lfloor y \sqrt{n} \rfloor}} \int_{-1}^{1} x^{n+k} P_{\lfloor y \sqrt{n} \rfloor}\left(x\right) \psi\left(x\right)  dx$ for various $k$, with coefficients summing to $1$. Hence, to prove \eqref{localk}, it suffices to establish the following limit:
    \begin{equation}
        \lim_{n \to \infty} \frac{\sqrt{n}}{\left(q;q\right)_{\lfloor y \sqrt{n} \rfloor}} \int_{-1}^{1} x^{n+k} P_{\lfloor y \sqrt{n} \rfloor}\left(x\right) \psi\left(x\right)  dx = \frac{1}{\sqrt{c\pi}} e^{-\frac{y^2}{4c}}. \label{iden}
    \end{equation}
    
    Indeed, \eqref{iden} holds because the additional constant $k$ does not affect the term $\left(1 - \frac{u^2}{2n}\right)^{n+k}$ as $n \to \infty$. Thus, it converges to the same limit as in the case $k=0$, by exactly the same argument used in Lemma \ref{lem.local0}.
    
    Similar to the proof of Theorem \ref{thm.uni}, from the local limit theorem \eqref{localk}, we obtain
    \[
    \lim_{n \to \infty} \frac{Q_{n+1}}{\sqrt{n}} \Big|_{Q_1 = k} \overset{\left(d\right)}{=} |\mathcal{N}\left(0,2c\right)|.
    \]
    To conclude, we still need to show the unconditional convergence of $\frac{Q_{n+1}}{\sqrt{n}}$. By the Portmanteau theorem for convergence in distribution, for every bounded continuous test function $\varphi$ and for each fixed $k$, we have
    \[
    \mathbb{E}\left[\varphi\left(\tfrac{Q_{n+1}}{\sqrt{n}}\right)\Big|Q_1=k\right] \xrightarrow[n\to\infty]{} \mathbb{E}\left[\varphi\left(|\mathcal{N}\left(0,2c\right)|\right)\right].
    \]
    Since $\pi_k$ is a probability measure and $\varphi$ is bounded, the dominated convergence theorem gives
    \[
    \mathbb{E}\left[\varphi\left(\frac{Q_{n+1}}{\sqrt{n}}\right)\right] = \sum_{k\ge 0}\pi_k\mathbb{E}\left[\varphi\left(\frac{Q_{n+1}}{\sqrt{n}}\right)\big|Q_1 = k\right] \xrightarrow[n\to\infty]{} \sum_{k\ge 0}\pi_k\mathbb{E}\left[\varphi\left(|\mathcal{N}\left(0,2c\right)|\right)\right] = \mathbb{E}\left[\varphi\left(|\mathcal{N}\left(0,2c\right)|\right)\right].
    \]
    By Portmenteau theorem again, this is equivalent to the convergence stated in \eqref{scaling}:
    \[
    \lim_{n \to \infty}\frac{Q_{n+1}}{\sqrt{n}} \overset{d}{=} |\mathcal{N}\left(0,2c\right)|.
    \]
    
    Now we turn to the second statement of the theorem. From \eqref{scaling}, together with the fact that the expectation of a folded Gaussian $|\mathcal{N}\left(0,\sigma\right)|$ is $\sigma\sqrt{\tfrac{2}{\pi}}$, and using Martin’s characterization of the convoy set \eqref{randomu}, the coupling argument \eqref{Ecoupling}, and the finiteness of $\mathbb{E}\left[Q_1\right]$ \eqref{Eq1}, we obtain
    \[
    \lim_{n \to \infty} \frac{\mathbb{E}_x\left[\mathcal{C}_0^{n}\right]}{\sqrt{n}} = \lim_{n \to \infty}\frac{\mathbb{E}_x\left[\mathcal{U}_{n}\right]}{\sqrt{n}} = \lim_{n \to \infty} \frac{\mathbb{E}_x\left[Q_{n+1} - Q_1\right]}{\sqrt{n}} = \lim_{n \to \infty} \frac{\mathbb{E}_x\left[Q_{n+1}\right]}{\sqrt{n}} = \sqrt{\frac{4x\left(1-x\right)}{\pi}}.
    \]

\end{proof}

\section{Weakly asymmetric limit of the convoy size}\label{qto1}
In this section, we study asymptotic result for the size of convoy as $q \to 1$. Throughout this section, we fix $z\in\mathbb{R}$ and set
\begin{equation}
    \label{standardnotation}q_n = e^{-\gamma/\sqrt{n}}, \qquad m_n\left(z\right) = \lfloor z\sqrt{n} + \tfrac{1}{\gamma}\sqrt{n}\log\sqrt{n}\rfloor.
\end{equation}

\begin{definition}
For $0<q<1$ and $z\in\mathbb{C}\setminus\{0,-1,-2,\ldots\}$, the $q$-Gamma function is
\begin{equation}
    \Gamma_q\left(z\right):= \left(1-q\right)^{1-z}\frac{\left(q;q\right)_\infty}{\left(q^{z};q\right)_\infty}.\label{qgamma}
\end{equation}

\end{definition}
\begin{lemma}
If $z \neq 0,-1,-2,\ldots$ then
\begin{equation}
\lim_{q \nearrow 1} \Gamma_{q}\left(z\right) = \Gamma\left(z\right).\label{qconverge}
\end{equation}
\end{lemma}

\begin{lemma}
Let $x\in \mathbb{R}$, $\gamma > 0$ and $n\in \mathbb{N}$. Let $q_n $ and $m_n\left(x\right)$ be as in our standard notation \eqref{standardnotation}. Then, 
\begin{equation}
    \lim_{n\to\infty} \left(q_n^{m_n\left(z\right)};q_n\right)_\infty = \exp\Big(-\tfrac{1}{\gamma} e^{-\gamma z}\Big).
    \label{den1}
\end{equation}
\end{lemma}
\begin{proof}
    It is sufficient to prove the claim for the non-integer quantity
    \[
    k_n\left(z\right) := \sqrt{n}z+\tfrac{1}{\gamma}\sqrt{n}\log\left(\sqrt{n}\right).
    \]  
    since the floor in the lemma changes $\log\left(q_n^{k_n};q_n\right)_\infty$ by only $o\left(1\right)$.
    We have
    \[
    \left(q_n^{k_n\left(z\right)};q_n\right)_\infty = \exp\left(\sum_{j=0}^\infty \log\bigl(1-q_n^{k_n\left(z\right)+j}\bigr)\right), \qquad y_{n,j} := q_n^{k_n\left(z\right)+j}.
    \]
    With $q_n = e^{-\gamma/\sqrt{n}}$,
    \[
    y_{n,j} = e^{-\left(\gamma/\sqrt{n}\right)\left(k_n\left(z\right)+j\right)} = e^{-\gamma z} e^{-\log\left(\sqrt{n}\right)} e^{-\gamma j/\sqrt{n}} = e^{-\gamma z}n^{-1/2}e^{-\gamma j/\sqrt{n}}.
    \]
    Hence $0< y_{n,j}\le e^{-\gamma z}/\sqrt{n}$. In particular, once $n$ is large enough we have $y_{n,j}\le \delta$ for all $j$, with some $\delta\in\left(0,1/2\right)$. Thus the $y_{n,j}$ are uniformly small.
    For $0\le y\le \delta$, 
    \[
    \log\left(1-y\right) = -y + R\left(y\right), \qquad |R\left(y\right)| \le C y^2.
    \]
    Therefore
    \[
    \sum_{j=0}^\infty \log\left(1-y_{n,j}\right) = - \sum_{j=0}^\infty y_{n,j} + O\left( \sum_{j=0}^\infty y_{n,j}^2 \right).
    \]
    Since $y_{n,j} = e^{-\gamma z}n^{-1/2}e^{-\gamma j/\sqrt{n}}$,
    \[
    \sum_{j=0}^\infty y_{n,j} = \frac{e^{-\gamma z}}{\sqrt{n}} \sum_{j=0}^\infty e^{-\gamma j/\sqrt{n}} = \frac{e^{-\gamma z}}{\sqrt{n}}\cdot \frac{1}{1-e^{-\gamma/\sqrt{n}}} = \frac{1}{\gamma} e^{-\gamma z} + o\left(1\right).
    \]
    Similarly, we can prove the error term is $O\left(\frac{1}{n}\right)$. Thus, 
    \[
    \left(q_n^{k_n\left(z\right)};q_n\right)_\infty = \exp\Big(-\tfrac{1}{\gamma} e^{-\gamma z} + o\left(1\right)\Big) \longrightarrow \exp\Big(-\tfrac{1}{\gamma} e^{-\gamma z}\Big).
    \]
\end{proof}
\begin{proposition}
Consider a sequence of random variables $X^n_0$ on $\mathbb{Z}^{+}$, whose distribution is defined by
\[
\mathbf{P}\left(X^n_0=k\right)=\frac{\left(q_n;q_n\right)_\infty}{\left(q_n;q_n\right)_k}q_n^k,\qquad k=0,1,2,\dots
\]
Then, as $n\to\infty$
\begin{equation}
    \frac{X^n_0-\tfrac{1}{\gamma}\sqrt{n}\log\left(\sqrt{n}\right)}{\sqrt{n}} \overset{d}{\Rightarrow} X,\label{x0weakly}
\end{equation}
where $X$ has density
\[
f_X\left(x\right)=e^{-\gamma x}\exp\Big(-\tfrac{1}{\gamma} e^{-\gamma x}\Big),\qquad x\in\mathbb{R}.
\]
\end{proposition}
\begin{remark}
    The initial distribution $\pi_k^{q}$ of $Q_1^{q}$ is defined in \eqref{lawofpi}, which is the same as $X_0^{n}$ if we choose $q = e^{-\gamma/\sqrt{n}}$.
\end{remark}
\begin{proof}
    Similar to the proof of Theorem \ref{thm.uni}, it is sufficient to prove the local limit theorem for $X_0^{n}$, i.e.
    \[
    \lim_{n \to \infty}\sqrt{n}\mathbb P\big(X^n_0=m_n\left(x\right)\big)=\lim_{n \to \infty}\sqrt{n}q_n^{m_n\left(x\right)}\frac{\left(q_n;q_n\right)_\infty}{\left(q_n;q_n\right)_{m_n\left(x\right)}} =e^{-\gamma x}\exp\Big(-\tfrac{1}{\gamma} e^{-\gamma x}\Big).
    \]
    By \eqref{den1}, $\lim_{n \to \infty}\frac{\left(q_n;q_n\right)_\infty}{\left(q_n;q_n\right)_{m_n\left(x\right)}} = \lim_{n \to \infty}\left(q_n^{m_n\left(x\right)};q_n\right)_\infty = \exp\Big(-\tfrac{1}{\gamma} e^{-\gamma x}\Big)$.
    We can write $m_n\left(x\right)=\sqrt{n}x+\tfrac{1}{\gamma}\sqrt{n}\log\left(\sqrt{n}\right)-\varepsilon_n$ with $\varepsilon_n\in[0,1)$. We conclude our proof by
    \[
    \sqrt{n}q_n^{m_n\left(x\right)}=\sqrt{n}\exp\Big(-\gamma x-\tfrac{\gamma}{\sqrt{n}}\tfrac{1}{\gamma}\sqrt{n}log\left(\sqrt{n}\right)+\tfrac{\gamma\varepsilon_n}{\sqrt{n}}\Big)= e^{-\gamma x}\left(1+o\left(1\right)\right)\Big).
    \]
\end{proof}
\begin{proposition}
Let $x \in \mathbb{R}$ and $u \geq 0$ be fixed. For each $n \in \mathbb{N}$, set $m_n\left(x\right)$ $q_n$ under our standard condition \eqref{standardnotation} and $x_n = \cos\left(\frac{u}{\sqrt{n}}\right)$. 
Then, for the continuous big $q$-Hermite polynomials \eqref{qhermite}, the following limit holds:
\begin{equation}
    \lim_{n\to\infty} H_{m_n\left(x\right)}\bigl(x_n\mid q_n\bigr) =\exp\left(\tfrac{1}{\gamma} e^{-\gamma x}\right)2\Re\left[\bigl(\gamma e^{\gamma x}\bigr)^{\tfrac{i u}{\gamma}}\frac{\Gamma\left(2iu/\gamma\right)}{\Gamma\left(iu/\gamma\right)}{}_1F_{1}\left(1 - \tfrac{iu}{\gamma};1-\tfrac{2iu}{\gamma};-\tfrac{1}{\gamma}e^{-\gamma x}\right)\right].\label{qherm1}
\end{equation}
Here ${}_1F_{1}$ denotes the confluent hypergeometric function, defined by the absolutely convergent series
\[
{}_1F_{1}\left(a;c;z\right)
=\sum_{k=0}^{\infty} \frac{\left(a\right)_k}{\left(c\right)_k}\frac{z^k}{k!},
\qquad \left(a\right)_k = a\left(a+1\right)\cdots\left(a+k-1\right).
\]
\end{proposition}
\begin{proof}
    By using generating function of $H\left(x|q\right)$ \eqref{Hgenerating} and Cauchy’s theorem with contour $|t|=q_n <1$, we have
    \[
    \frac{H_{m_n\left(x\right)}\left(\cos \frac{u}{\sqrt{n}} \mid q_n\right)}{\left(q_n;q_n\right)_{m_n\left(x\right)}} = \frac{1}{2\pi i} \oint_{|t|=q_n} \frac{\left(t;q_n\right)_\infty}{\left(q^{iu/\gamma+z};q\right)_\infty \left(q^{-iu/\gamma+z};q_n\right)_\infty} \frac{dt}{t^{m_n\left(x\right)+1}}.
    \]

    With the change of variables
    \[
    t = q^z_n, \qquad dt = q^z \ln\left(q\right)\; dz, \qquad z \in [1-i\frac{\pi\sqrt{n}}{\gamma},1+i\frac{\pi\sqrt{n}}{\gamma}),
    \]
    we have

    \[
    H_{m_n\left(x\right)}\left(\cos \frac{u}{\sqrt{n}} \mid q_n\right) = \left(q_n;q_n\right)_{m_n\left(x\right)} \frac{1}{2\pi i} \frac{\gamma}{\sqrt{n}}\int_{1-i\frac{\pi\sqrt{n}}{\gamma}}^{1+i\frac{\pi\sqrt{n}}{\gamma}} \frac{\left(q^z_n;q_n\right)_\infty}{\left(q_n^{iu/\gamma+z};q_n\right)_\infty \left(q_n^{-iu/\gamma+z};q_n\right)_\infty} q_n^{-m_{n}\left(x\right)z}\; dz.
    \]
    Using $m_n\left(z\right) = \lfloor x\sqrt{n} + \tfrac{1}{\gamma}\sqrt{n}\log\sqrt{n}\rfloor$ and the definition of the q-Gamma function \eqref{qgamma}, we rewrite $H_{m_n}$ as
    \[
    H_{m_n\left(x\right)}\left(x_n\mid q_n\right) = \frac{\left(q_n;q_n\right)_{m_n\left(x\right)}}{\left(q_n;q_n\right)_{\infty}}\frac{1}{2\pi i}\int_{1-i\infty}^{1+i\infty} F_n\left(z\right)\; dz,
    \]
    with
    \[
    F_n\left(z\right) =\mathbf{1}_{\{\Re z = 1, |\Im z|\le \tfrac{\pi \sqrt{n}}{\gamma}\}} \frac{\Gamma_{q_n}\Big(z+\tfrac{i u}{\gamma}\Big) \Gamma_{q_n}\Big(z-\tfrac{i u}{\gamma}\Big)}{\Gamma_{q_n}\left(z\right)} \frac{e^{\gamma zx}}{\sqrt{n}^{1-z}\left(1-q_n\right)^{1-z}}.
    \]
    The prefactor $\frac{\left(q_n;q_n\right)_{m_n\left(x\right)}}{\left(q_n;q_n\right)_{\infty}}$ converges to $\exp\left(\frac{1}{\gamma}e^{-\gamma x}\right)$ due to \eqref{den1}. The pointwise convergence of $F_n\left(z\right)$ is immediate  due to \eqref{qconverge} and the definition of $q_n$. We have
    \[
    \lim_{n\to \infty}F_n\left(z\right) = \mathbf{1}_{\{\Re z = 1\}}\left(\gamma e^{\gamma x}\right)^{z}\frac{\Gamma\Big(z+\tfrac{i u}{\gamma}\Big) \Gamma\Big(z-\tfrac{i u}{\gamma}\Big)}{\Gamma\left(z\right)}.
    \]
    To bound $F_n\left(z\right)$ uniformly, we prove that there exists $N$ such that for any $\sqrt{n} > N$, $\Big|\frac{\Gamma_{q_n}\bigl(1+i\left(s+\tfrac{u}{\gamma}\right)\bigr) \Gamma_{q_n}\bigl(1+i\left(s-\tfrac{u}{\gamma}\right)\bigr)}{\Gamma_{q_n}\left(1+is\right)}\Big|$ has exponential decay with respect to $s$. It is sufficient to prove it for $\gamma = 2$, since the general $\gamma$ case follows by taking $N^{\left(\gamma\right)} = \frac{\gamma}{2} N^{\left(2\right)}$.
    From \cite[(5.17)]{bryc2025limits}, there exists a $N^*$ such that for any $N>N^*$, $q = e^{-2/\sqrt{n}}$ for all $x \in [-\frac{\pi \sqrt{n}}{2}, \frac{\pi \sqrt{n}}{2}]$.
    \begin{equation}
    c_1 e^{x^{2}/\sqrt{n}}\frac{\pi |x|}{|\sinh\left(\pi x\right)|}<|\Gamma_{q}\left(1+ix\right)|^{2}< c_2 e^{x^{2}/\sqrt{n}}\frac{\pi |x|}{|\sinh\left(\pi x\right)|}.\label{qgammabound}
    \end{equation}
    We have the following exponential decay with respect to $s$:
    \begin{equation}
    \Big|\frac{\Gamma_{q_n}\bigl(1+i\left(s+\tfrac{u}{2}\right)\bigr) \Gamma_{q_n}\bigl(1+i\left(s-\tfrac{u}{2}\right)\bigr)}{\Gamma_{q_n}\left(1+is\right)}\Big| \leq C\left(u\right) e^{s^{2}/2M}\sqrt{\frac{\pi |s|}{|\sinh\left(\pi s\right)|}} \leq C\left(u\right)e^{-\frac{3|s|}{20}}.\label{expdecay}
    \end{equation}
    The last inequality is due to $x \in [-\frac{\pi \sqrt{n}}{2}, \frac{\pi \sqrt{n}}{2}]$, $\frac{x^2}{2\sqrt{n}} < \frac{\pi|x|}{4}$ and the elementary bound $\frac{|x|}{|\sinh\left(x\right)|} \leq C e^{-4|x|/5}$.
    
    Equation \eqref{expdecay} allows us to apply the dominated convergence theorem to $F_n\left(z\right)$. We obtain
    \begin{equation}
        \lim_{n\to\infty} \exp\Big(-\tfrac{1}{\gamma} e^{-\gamma x}\Big)H_{m_n}\bigl(x_n\mid q_n\bigr)= \tfrac{1}{2\pi i}\int_{1-i\infty}^{1+i\infty} \left(\gamma e^{\gamma x}\right)^{z}\frac{\Gamma\Big(z+\tfrac{i u}{\gamma }\Big) \Gamma\Big(z-\tfrac{i u}{\gamma }\Big)}{\Gamma\left(z\right)} \; dz := \mathcal{J}\left(x,u\right). 
    \end{equation}
    
    It remains to prove that
    \[
    \mathcal{J}\left(x,u\right) = 2\Re\left[\bigl(\gamma e^{\gamma x}\bigr)^{\tfrac{i u}{\gamma}}\frac{\Gamma\left(2iu/\gamma\right)}{\Gamma\left(iu/\gamma\right)}{}_1F_{1}\left(1 - \tfrac{iu}{\gamma};1 - \tfrac{2iu}{\gamma};-\tfrac{1}{\gamma}e^{-\gamma x}\right)\right].
    \]
    Define \(\eta:= \left(\gamma e^{\gamma x}\right)^{-1}=\frac{1}{\gamma}e^{-\gamma x}\) so that \(\left(\gamma e^{\gamma x}\right)^{z}=\eta^{-z}\) and \(\eta>0\) for real \(x\). We shift the contour leftwards. The poles crossed are exactly those of \(\Gamma\left(z\pm iu/\gamma\right)\), namely
    \[
    z=-\tfrac{iu}{\gamma}-n \quad\text{and}\quad z=+\tfrac{iu}{\gamma}-n, \qquad n=0,1,2,\dots.
    \]
By Cauchy's residue theorem, \(\mathcal{J}\left(x,u\right)\) equals the sum of the residues at these poles. For \(z=-\tfrac{iu}{\gamma}-n\), we have
\[
\operatorname{Res}_{z=-\frac{iu}{\gamma}-n} 
\frac{\Gamma\left(z+\tfrac{iu}{\gamma}\right)\Gamma\left(z-\tfrac{iu}{\gamma}\right)}{\Gamma\left(z\right)}\eta^{-z}
=
\frac{\left(-1\right)^n}{n!}
\frac{\Gamma\left(-\tfrac{2iu}{\gamma}-n\right)}{\Gamma\left(-\tfrac{iu}{\gamma}-n\right)}
\eta^{\frac{iu}{\gamma}+n}.
\]
Using \(\Gamma\left(\alpha-n\right)=\Gamma\left(\alpha\right)\left(-1\right)^n/\left(1-\alpha\right)_n\) yields
\[
\sum_{n=0}^{\infty}\frac{\left(-1\right)^n}{n!}
\frac{\Gamma\left(-\tfrac{2iu}{\gamma}-n\right)}{\Gamma\left(-\tfrac{iu}{\gamma}-n\right)}
\eta^{\frac{iu}{\gamma}+n}= \eta^{\frac{iu}{\gamma}}
\frac{\Gamma\left(-\tfrac{2iu}{\gamma}\right)}{\Gamma\left(-\tfrac{iu}{\gamma}\right)}
{}_1F_{1}\left(1+\tfrac{iu}{\gamma};1+\tfrac{2iu}{\gamma};-\eta\right).
\]
Similarly, for \(z=\tfrac{iu}{\gamma}-n\),
\[
\sum_{n=0}^{\infty}\frac{\left(-1\right)^n}{n!}
\frac{\Gamma\left(\tfrac{2iu}{\gamma}-n\right)}{\Gamma\left(\tfrac{iu}{\gamma}-n\right)}
\eta^{-\frac{iu}{\gamma}+n}=\eta^{-\frac{iu}{\gamma}}
\frac{\Gamma\left(\tfrac{2iu}{\gamma}\right)}{\Gamma\left(\tfrac{iu}{\gamma}\right)}
{}_1F_{1}\left(1-\tfrac{iu}{\gamma};1-\tfrac{2iu}{\gamma};-\eta\right).
\]
Adding both residue families gives
\begin{align*}
\mathcal{J}\left(x,u\right)
&=
\eta^{\frac{iu}{\gamma}}
\frac{\Gamma\left(-\tfrac{2iu}{\gamma}\right)}{\Gamma\left(-\tfrac{iu}{\gamma}\right)}
{}_1F_{1}\left(1+\tfrac{iu}{\gamma};1+\tfrac{2iu}{\gamma};-\eta\right)
+
\eta^{-\frac{iu}{\gamma}}
\frac{\Gamma\left(\tfrac{2iu}{\gamma}\right)}{\Gamma\left(\tfrac{iu}{\gamma}\right)}
{}_1F_{1}\left(1-\tfrac{iu}{\gamma};1-\tfrac{2iu}{\gamma};-\eta\right).
\end{align*}
For real \(x,u\) the two terms are complex conjugates, hence an equivalent real form is
\[
\mathcal{J}\left(x,u\right)=2\Re\left[
\eta^{-\frac{iu}{\gamma}}
\frac{\Gamma\left(2iu/\gamma\right)}{\Gamma\left(iu/\gamma\right)}
{}_1F_{1}\left(1-\tfrac{iu}{\gamma};1-\tfrac{2iu}{\gamma};-\eta\right)
\right].
\]
Since \(\eta^{-iu/\gamma}=\left(\gamma e^{\gamma x}\right)^{iu/\gamma}\) and
\(\eta=\frac{1}{\gamma}e^{-\gamma x}\), this is exactly
\[
\mathcal{J}\left(x,u\right) = 2\Re\left[\bigl(\gamma e^{\gamma x}\bigr)^{\tfrac{i u}{\gamma}}\frac{\Gamma\left(2iu/\gamma\right)}{\Gamma\left(iu/\gamma\right)}{}_1F_{1}\left(1 - \tfrac{iu}{\gamma};1 - \tfrac{2iu}{\gamma};-\tfrac{1}{\gamma}e^{-\gamma x}\right)\right],
\]
as claimed.
\end{proof}
\begin{proposition}
    Let $q_n=e^{-\gamma/\sqrt{n}}$. We have
    \begin{equation}
        \lim_{n \to \infty}\left(\frac{Q_1 - \lfloor\tfrac12\sqrt{n}\log\sqrt{n}\big\rfloor}{\sqrt{n}},\frac{Q_{n+1} - \lfloor\tfrac12\sqrt{n}\log\sqrt{n}\big\rfloor}{\sqrt{n}} \right) \overset{d}{=} \left(X^\gamma,Y^\gamma\right), \label{weaklyjoint}
    \end{equation}
    where the random vector $\left(X^\gamma,Y^\gamma\right)$ has density \eqref{jointdensity}.
\end{proposition}
\begin{proof}
    Since the distribution of $Q_1$ in the limit is proved in \eqref{x0weakly}, we only have to prove the conditionally law of $Q_{n+1}$ under $Q_1$. For the same reason as in the proof Theorem \ref{thm.uni}, it is sufficient to prove
    \[
    \lim_{n \to \infty}\sqrt{n}\mathbb{P}\left(Q_{n+1} = m_n\left(y\right)\big| Q_1 = m_n\left(x\right)\right)\ = \frac{1}{2\pi}\exp\Big(\tfrac{1}{\gamma} e^{-\gamma x}\Big)\int_{0}^{\infty} e^{-c u^2}\frac{|\Gamma\left(iu/\gamma\right)|^2}{|\Gamma\left(2iu/\gamma\right)|^2}\mathcal{J}\left(x,u\right)\mathcal{J}\left(y,u\right)du,
    \]
    where $m_n\left(z\right)=\big\lfloor z\sqrt{n}+\tfrac{1}{\gamma}\sqrt{n}\log\sqrt{n}\big\rfloor$.
    We apply \eqref{Htrans} to $i = m_n\left(x\right)$, $j = m_n\left(y\right)$,
    \[
    \sqrt{n}\mathbb{P}\left(Q_{n+1} = m_n\left(y\right)| Q_1 = m_n\left(x\right)\right) = \frac{\sqrt{n}}{\left(q;q\right)_{m_n\left(y\right)}}\int_{-1}^{1} \left(2cz + 1 - 2c\right)^n H_{m_n\left(x\right)}\left(z\right)H_{m_n\left(y\right)}\left(z\right) w\left(z\right) \; dz.
    \]
    For exactly the same reason as the proof of Theorem \ref{thm.uni}, the integral 
    \[
    \int_{-1}^{1} \left(2cz + 1 - 2c\right)^n H_{m_n\left(x\right)}\left(z\right)H_{m_n\left(y\right)}\left(z\right) w\left(z\right) \; dz
    \]
    is negligible on the interval $\left[-1,0\right]$. We perform a change of variable $z = \cos\left(u/\sqrt{n}\right)$. We have
    \[
    \sqrt{n}\mathbb{P}\left(Q_{n+1} = m_n\left(y\right)| Q_1 = m_n\left(x\right)\right)= \int_{0}^{\frac{\pi}{2}\sqrt{n}} G_n\left(u\right)du,
    \]
    where the integrand $G_n\left(u\right)$ can be written as
    \[
    G_n\left(u\right):=\Bigl(1-2c\left(1-\cos\tfrac{u}{\sqrt n}\right)\Bigr)^{n}H_{m_n\left(x\right)}\Bigl(\cos\tfrac{u}{\sqrt{n}}\Bigr)H_{m_n\left(y\right)}\Bigl(\cos\tfrac{u}{\sqrt{n}}\Bigr)\frac{\sin\left(u/\sqrt n\right)w\left(\cos\left(u/\sqrt n\right)\right)}{\left(q_n;q_n\right)_{m_n\left(y\right)}}.
    \]
    
    We first prove pointwise limit of $G_n$. From the elementary limit and \eqref{qherm1}, for $u > 0$, we have
    \[
    \Bigl(1-2c\left(1-\cos\tfrac{u}{\sqrt n}\right)\Bigr)^n\to e^{-c u^2}, \qquad H_{m_n\left(z\right)}\Bigl(\cos\tfrac{u}{\sqrt{n}}\Bigr) \to \exp\Big(\tfrac{1}{\gamma} e^{-\gamma z}\Big)\mathcal{J}\left(z,u\right)\qquad z\in\{x,y\}.
    \]
    Using \eqref{refdef1}, we have the exact identity
    \begin{equation}
    w\left(\cos\frac{u}{\sqrt{n}}\right) \sin\left(\tfrac{u}{\sqrt{n}}\right) =\frac{\left(q_n;q_n\right)_\infty}{2\pi} \frac{\bigl|\left(q_n^{2iu/\gamma};q_n\right)_\infty\bigr|^{2}}{\bigl|\left(q_n^{iu/\gamma};q_n\right)_\infty\bigr|^{2}}.
    \end{equation}
By \eqref{den1},
\[
\frac{\left(q_n;q_n\right)_\infty}{\left(q_n;q_n\right)_{m_n\left(y\right)}}=\left(q_n^{m_n\left(y\right)+1};q_n\right)_\infty
\to \exp\left(-\tfrac{1}{\gamma} e^{-\gamma y}\right).
\]
By the $q$--Gamma identity \eqref{qgamma}, we obtain
\[
\frac{\left(q_n^{2iu/\gamma};q_n\right)_\infty}{\left(q_n^{iu/\gamma};q_n\right)_\infty}
=\frac{\Gamma_{q_n}\left(iu/\gamma\right)}{\Gamma_{q_n}\left(2iu/\gamma\right)}\left(1-q_n\right)^{-iu/\gamma}.
\]
Taking absolute values, the $\left(1-q_n\right)^{-iu/2}$ factor drops out, so
\[
\frac{|\left(q_n^{iu};q_n\right)_\infty|}{|\left(q_n^{iu/2};q_n\right)_\infty|}
=\frac{|\Gamma_{q_n}\left(iu/\gamma\right)|}{|\Gamma_{q_n}\left(2iu/\gamma\right)|}.
\]
As $q_n\to 1$, the $q$-Gamma function converges to the classical Gamma function \eqref{qconverge}. Therefore,
\[
\lim_{n\to\infty}\frac{|\left(q_n^{iu};q_n\right)_\infty|^2}{|\left(q_n^{iu/2};q_n\right)_\infty|^2}
=\frac{|\Gamma\left(iu/\gamma\right)|^2}{|\Gamma\left(2iu/\gamma\right)|^2}.
\]
Thus, we obtain the following pointwise convergence of $G_n$:
\begin{align*}
    \lim_{n \to \infty}G_n &= \frac{1}{2\pi} \int_{0}^{\infty} e^{-c u^2}\exp\left(-\tfrac{1}{\gamma} e^{-\gamma y}\right) \frac{|\Gamma\left(iu/\gamma\right)|^2}{|\Gamma\left(2iu/\gamma\right)|^2} \exp\left(\tfrac{1}{\gamma} e^{-\gamma x}\right)\mathcal{J}\left(x,u\right) \exp\left(\tfrac{1}{\gamma} e^{-\gamma y}\right)\mathcal{J}\left(y,u\right)du\\
    &= = \frac{1}{2\pi} \exp\left(\tfrac{1}{\gamma} e^{-\gamma x}\right) \int_{0}^{\infty} e^{-c u^2} \frac{|\Gamma\left(iu/\gamma\right)|^2}{|\Gamma\left(2iu/\gamma\right)|^2} \mathcal{J}\left(x,u\right) \mathcal{J}\left(y,u\right)du.
\end{align*}

    Now, we turn to the bound of $G_n$. Clearly, for $n$ sufficiently large
    \[
    \Bigl(1-2c\left(1-\cos\tfrac{u}{\sqrt n}\right)\Bigr)^n < Ce^{-c u^2}.
    \]
    We need the following lemma for the upper bound of orthogonal polynomial terms.
    \begin{lemma}
    Fix $z\in\mathbb{R}$. Set $q_n$ and $m_n\left(z\right)$ under standard condition \eqref{standardnotation} and
    \[
    x=\cos\theta \qquad \theta \in [0,\frac{\pi}{2}].
    \]
    Then there exists a constant $A$ such that, for $n\ge1$,
    \[
    \bigl|H_{m_n\left(z\right)}\left(x\mid q_n\right)\bigr|\ \le\ A.
    \]
\end{lemma}

\begin{proof}
From the generating function \eqref{Hgenerating} and  Cauchy's Integral Formula, we express the normalized polynomial as a contour integral
\begin{equation}
    \frac{H_{m_n}(x \mid q_n)}{(q_n;q_n)_{m_n}} = \frac{1}{2\pi i} \oint_{\mathcal{C}} \frac{G_{\theta}(t;q_n)}{t^{m_n+1}} \, dt,
\end{equation}
where $\mathcal{C}$ is a small circle around the origin inside the unit disk and
\[
G_\theta\left(t;q\right):=\frac{\left(t;q\right)_\infty}{\left(e^{i\theta}t;q\right)_\infty\left(e^{-i\theta}t;q\right)_\infty}=\prod_{k=0}^{\infty}\frac{1-tq^k}{\left(1-e^{i\theta}tq^k\right)\left(1-e^{-i\theta}tq^k\right)}.
\]
The function $G_\theta\left(t;q\right)$ is meromorphic with simple poles at $t_{k,\pm} = q_n^{-k}e^{\pm i\theta}$ for $k \in \mathbb{N}_0$. We deform the contour $\mathcal{C}$ to infinity, which allows us to evaluate the integral as the negative sum of the residues at these poles. Specifically,
\[
\left|H_{m_n}(x \mid q_n)\right| \le (q_n;q_n)_{m_n} \sum_{i = 0}^{\infty}\left( \left|\text{Res}(t_{k,+})\right| + \left|\text{Res}(t_{k,-})\right|\right).
\]
Due to the symmetry of the complex conjugate poles, we focus on the summation over $t_{k,+}$, noting that computations for $t_{k,-}$ terms proceed identically. Utilizing the identity $(a;q)_\infty = (a;q)_k (1-aq^k) (aq^{k+1};q)_\infty$, we find that at $t = t_{k,+}$, the residue satisfies
\begin{equation}
    \text{Res}(t_{k,+}) = C_k  \frac{(e^{i \theta};q_n)_\infty}{(e^{2i \theta};q_n)_\infty(q_n;q_n)_\infty}t_{k,+}^{-m_n + 1},
\end{equation}
where 
\[
C_k := \frac{(e^{i \theta}q_n^{-k};q_n)_{k+1}}{(e^{2i \theta}q_n^{-k};q_n)_{k+1}(q_n^{-k};q_n)_{k}} = \frac{(e^{i \theta};q_n)_{k+1}}{(e^{2i \theta};q_n)_{k+1}} q_n^{k(k+1)/2}\frac{1}{(q_n;q_n)_k}.
\]
Since $\theta \in [0,\pi/2]$, we have the elementary bound for any $q \in [0,1)$,
\[
\left|\frac{1 - e^{i \theta}q^j}{1 - e^{2i \theta}q^j}\right| \leq 1, \quad \text{hence} \qquad \left|\frac{(e^{i \theta};q)_\infty}{(e^{2i \theta};q)_\infty}\right| \leq 1, \qquad \left|\frac{(e^{i \theta};q)_{k+1}}{(e^{2i \theta};q)_{k+1}} q^{k(k+1)/2}\right|\leq 1.
\]
Thus, 
\[
\left| \text{Res}(t_{k,+})\right| \leq \frac{1}{(q_n;q_n)_k} t_{k,+}^{-m_n} \frac{1}{(q_n;q_n)_\infty}.
\]
We now substitute the scaling limits for $q_n$ and $m_n$ as $n \to \infty$. We have $q_n = e^{-\gamma/\sqrt{n}}$ and $m_n = \lfloor z\sqrt{n} + \frac{1}{\gamma}\sqrt{n}\log \sqrt{n} \rfloor$. First, we estimate the residue denominator $(q_n; q_n)_k$. As $q_n \to 1$, we have $(1 - q_n^j) \sim j\gamma/\sqrt{n}$. Thus,
\begin{equation}
    \frac{1}{(q_n; q_n)_k} \sim \frac{1}{\prod_{j=1}^k (j\gamma/\sqrt{n})} = \frac{n^{k/2}}{k! \gamma^k}.
\end{equation}
Second, we estimate the factor $t_{k,+}^{-m_n}$. Substituting the expressions for $q_n$ and $m_n$, we obtain
\[
|t_{k,+}^{-m_n}| \approx e^{-k\gamma z} n^{-k/2}.
\]
Finally, recovering $H_{m_n}$, we note that the prefactor $(q_n; q_n)_{m_n} / (q_n; q_n)_\infty$ converges to a finite constant due to \eqref{den1}. We have 
\begin{equation}
    |H_{m_n}(x \mid q_n)| \le C \sum_{k=0}^\infty \frac{1}{k!} \left( \frac{e^{-\gamma z}}{\gamma} \right)^k < \infty.
\end{equation}
Since the series converges, the sequence is bounded uniformly in $n$.
\end{proof}

From \cite[(5.24),(5.25)]{bryc2025limits}, there exists $A,B,N^{\left(2\right)}$, such that for all $n>N^{\left(2\right)}$, the following bound holds for $\gamma = 2$
\[
\mathbf{1}_{\{|u|\le \tfrac{\pi \sqrt{n}}{\gamma}\}}\frac{|\Gamma_{q_n}\left(iu/\gamma\right)|}{|\Gamma_{q_n}\left(2iu/\gamma\right)|} < A u e^{Bu}.
\]
In our case the same bound follows from taking $N^{\left(\gamma\right)} = \frac{\gamma}{2} N^{\left(2\right)}$. Combining all the previous result, we showed that $G_n$ can be bounded above by a integrable function on $[0,\infty)$. We conclude by using the dominated convergence theorem. 
\end{proof}

\begin{theorem}\label{weakly}
    Let the speed of particle $0$ be $u$, and define \(c = \frac{1-u^{2}}{4}, q_{n} = e^{-\tfrac{\gamma}{\sqrt{n}}}.\)
    Then
    \[
    \lim_{n \to \infty}\mathbb{E}\left[\frac{\# \mathcal{C}_0^{n,\left(q_n\right)}}{\sqrt{n}}\middle|
    U_0 = u\right]= \mathbb{E}\left[Y^\gamma - X^\gamma\right].
    \]
\end{theorem}
\begin{proof}
    Using the coupling argument \eqref{Ecoupling} and the joint asymptotic distribution \eqref{weaklyjoint} of $Q_1$ and $Q_{n+1}$, we obtain
    \[
    \lim_{n \to \infty}\mathbb{E}\left[\frac{\# \mathcal{C}_0^{n,\left(q_n\right)}}{\sqrt{n}}\middle|
    U_0 = u\right]= \lim_{n \to \infty}\mathbb{E}\left[\frac{Q_{n+1} - Q_1}{\sqrt{n}} \big| U_0 = u\right] = \mathbb{E}\left[Y^\gamma - X^\gamma\right].
    \]
\end{proof}

We investigate the limit of the joint distribution of $\left(X^\gamma, Y^\gamma\right)$ as $\gamma \to 0$ and $\gamma \to \infty$ respectively. 
\begin{proposition}
    Let $\left(X^\gamma, Y^\gamma\right)$ be the limiting distribution that we derived from the weakly asymmetric limit of the convoy size with density: 
    \[
    f^{\gamma}_{X,Y}\left(x,y\right) = \frac{1}{2\pi}e^{-\gamma x}\int_{0}^{\infty}e^{-c w^{2}}\frac{|\Gamma\left(iw/\gamma\right)|^2}{|\Gamma\left(2iw/\gamma\right)|^2}\mathcal{J}\left(x,w\right)\mathcal{J}\left(y,w\right)dw.
    \]
    When $\gamma \to \infty$, we have 
    \[
    X^{\gamma} \overset{d}{\Rightarrow} \delta_0 \qquad Y^{\gamma} \overset{d}{\Rightarrow} |\mathcal{N}\left(0, 2c\right)|,
    \]
    where $\mathcal{N}\left(0, 2c\right)$ is the normal distribution with mean $0$ and variance $2c$.
    When $\gamma \to 0$, we have 
    \[
    \lim_{\gamma \to 0}\mathbb{E}\left[Y^{\gamma} - X^{\gamma}\right] = 0.
    \]
\end{proposition}
\begin{remark}
    It turns out that when $\gamma \to \infty$, the limiting distribution of $\left(X^\gamma, Y^\gamma\right)$ converges to $\delta_{x=0} \frac{1}{\sqrt{c\pi}}e^{-\frac{y^2}{4c}}$, which is the limit we derived for fixed $q \in [0,1)$. This implies if $q_n = 1 - w\left(\frac{1}{\sqrt{n}}\right)$ ($w\left(\frac{1}{\sqrt{n}}\right)$ means it converges to $0$ much slower than $\frac{1}{\sqrt{n}}$), then limiting behavior of the expected size of the convoy is the same as the fixed $q$ case. 

    Moreover, when $\gamma \to 0$, both $X^\gamma$ and $Y^\gamma$ have infinite means. However, the difference $\mathbb{E}\left[Y^\gamma - X^\gamma\right]$ goes to $0$ as $\gamma \to 0$. This corresponds to the case where $q_n$ tends to 1 much faster than $1 - \frac{\gamma}{\sqrt{n}}$. Our asymptotic result implies the convoy set disappear at scale $\sqrt{n}$ if $q_n \to 1$ too fast. 
    
    Combining the two limit regimes of $\gamma$, we conclude that $q_n = 1 - \frac{\gamma}{\sqrt{n}}$ is the critical scaling for the convoy set existing at scale $\sqrt{n}$.
\end{remark}
\begin{proof}
    We first prove the case $\gamma \to 0$. Define a symmetric kernel
    \[
    K_\gamma\left(x,y\right) := \frac{1}{2\pi}\int_{0}^{\infty} e^{-c w^{2}} \frac{|\Gamma\left(iw/\gamma\right)|^2}{|\Gamma\left(2iw/\gamma\right)|^2} \mathcal{J}\left(x,w\right)\mathcal{J}\left(y,w\right) dw,
    \]
    so that
    \[
    f^{\left(\gamma\right)}_{X,Y}\left(x,y\right) = e^{-\gamma x}K_\gamma\left(x,y\right).
    \]
    For small gamma we have the elementary asymptote $e^{-\gamma x} = 1 - \gamma x + o\left(\gamma\right)$, which allows us to write the expectation as
    \[
    \mathbb{E}\left[Y^\gamma - X^\gamma\right] = \iint_{\mathbb{R}^2}\left(y-x\right)K_\gamma\left(x,y\right)dxdy - \gamma \iint \left(y-x\right)xK_\gamma\left(x,y\right)dxdy + o\left(\gamma\right).
    \]
    Since $K_\gamma\left(x,y\right)$ is symmetric in $x,y$, for any $\gamma$
    \[
    \iint_{\mathbb{R}^2}\left(y-x\right)K_\gamma\left(x,y\right)dxdy = 0.
    \]
    It remains to bound from above for $\gamma \iint \left(y-x\right)xK_\gamma\left(x,y\right)dxdy$ so that 
    \[
    \lim_{\gamma \to 0}\gamma \iint \left(y-x\right)xK_\gamma\left(x,y\right)dxdy = 0.
    \]
    Note that 
    \[
    \mathcal{J}\left(x,u\right) \leq |\gamma^{iu/\gamma}|\frac{|\Gamma\left(2iu/\gamma\right)|}{|\Gamma\left(iu/\gamma\right)|} |{}_1F_{1}\left(
        1 - \tfrac{iw}{\gamma};
        1 - \tfrac{2iw}{\gamma};
        -\tfrac{1}{\gamma} e^{-\gamma x}
    \right)|.
    \]
    $\frac{|\Gamma\left(2iu/\gamma\right)|}{|\Gamma\left(iu/\gamma\right)|}$ in $\mathcal{J}\left(x,u\right)$ and in $\mathcal{J}\left(y,u\right)$ will cancel out the prefactor $\frac{|\Gamma\left(iu/\gamma\right)|^2}{|\Gamma\left(2iu/\gamma\right)|^2}$ in $K_{\gamma}\left(x,y\right)$. So we only have to bound the hypergeometric function term. It has exponential decay with respect to $x, y, \gamma$ by using the Kummer transformations for confluent hypergeometric function \cite[13.1.27]{abramowitz1965handbook}. 
    \[
    {}_1F_1\left(a;b;z\right) = e^z {}_1F_1\left(b-a;b;-z\right),
    \]
    with $z = -\tfrac{1}{\gamma} e^{-\gamma x}$, $a = 1 - \tfrac{iw}{\gamma}$, $b = 1 - \tfrac{2iw}{\gamma}$. Now we prove the case when $\gamma \to \infty$. The distribution of $X^{\gamma}$ is explicit \eqref{x0weakly}. By computing the cumulative distribution function $F^\gamma\left(x\right)$ of $X^{\gamma}$, we obtain $F^\gamma\left(x\right) \to 0$ for any $x<0$ and $F^\gamma\left(x\right) \to 1$ for any $x>0$. Hence $X^\gamma \to \delta_{x = 0}$ in distribution. 

    For the marginal distribution of $Y^\gamma$consider
    \[
    f^{\gamma}_Y\left(y\right)=\int_{-\infty}^{\infty} f^{\gamma}_{X,Y}\left(x,y\right)dx.
    \]
    By Fubini theorem, we can write $f^{\gamma}_Y\left(y\right)$ as 
    \[
    f^{\gamma}_Y\left(y\right)=\frac{1}{2\pi}\int_{0}^{\infty}e^{-c w^{2}}\frac{|\Gamma\left(iw/\gamma\right)|^2}{|\Gamma\left(2iw/\gamma\right)|^2} \left( \int_{-\infty}^{\infty} e^{-\gamma x} \mathcal{J}\left(x,w\right) dx \right) \mathcal{J}\left(y,w\right)dw.
    \]
    To justify the asymptotic behavior of the $x$–integral, define
    \[
    I_\gamma\left(w\right):= \int_{-\infty}^{\infty} e^{-\gamma x}\mathcal{J}\left(x,w\right)dx.
    \]
    Substituting this into $I_\gamma\left(w\right)$ and performing the change of variables
    \[
    t = \frac{1}{\gamma}e^{-\gamma x},\qquad e^{-\gamma x} = \gamma t,\qquad e^{\gamma x} = \frac{1}{\gamma t},\qquad dx = -\frac{dt}{\gamma t},
    \]
    one obtains after a short computation
    \[
    I_\gamma\left(w\right) = 2\Re\left[\frac{\Gamma\left(2\varepsilon\right)}{\Gamma\left(\varepsilon\right)}\int_0^\infty t^{-\varepsilon}{}_1F_1\bigl(1-\varepsilon;1-2\varepsilon;-t\bigr)dt\right],
    \qquad \varepsilon = \frac{i w}{\gamma}.
    \]
    At $\varepsilon=0$ we use ${}_1F_1\left(1;1;-t\right)=e^{-t}$ and obtain
    \[
    I_\gamma\left(0\right) = 2 \cdot \frac{1}{2} \int_0^\infty e^{-t}dt = 1 \Rightarrow \lim_{\gamma \to \infty}I_\gamma\left(w\right) = 1,
    \]
    for any fixed $w$.
    The expansion $\Gamma\left(z\right)\sim 1/z$ as $z\to 0$ gives
    \[
    \frac{|\Gamma\left(iw/\gamma\right)|^{2}}{|\Gamma\left(2iw/\gamma\right)|^{2}}
    \sim \frac{|\gamma/\left(iw\right)|^{2}}{|\gamma/\left(2iw\right)|^{2}}=4,
    \qquad \gamma\to\infty.
    \]
    Thus the gamma ratio tends to the constant $4$.
    The only nontrivial dependence on $y$ comes from $\mathcal{J}\left(y,w\right)$.
    Writing $\varepsilon = iw/\gamma$, one has
    \[
    \left(\gamma e^{\gamma y}\right)^{\varepsilon} = \exp\big( \varepsilon\left(\log\gamma + \gamma y\right) \big) = \exp\big(i\tfrac{w}{\gamma}\log\gamma\big)e^{iwy},
    \]
    so the leading oscillatory behavior of $\mathcal{J}\left(y,w\right)$ is simply
    $e^{iwy}$.  The remaining factors in $\mathcal{J}$ converge as
    $\gamma\to\infty$: 
    $\Gamma\left(2\varepsilon\right)/\Gamma\left(\varepsilon\right)\to \tfrac12$, and the confluent
    hypergeometric term tends to $1$ when $y>0$ (since
    $\frac{1}{\gamma}e^{-\gamma y}\to 0$), but decays to $0$ when $y<0$ (since
    $\frac{1}{\gamma}e^{-\gamma y}\to\infty$). Taking real parts gives the
    asymptotic form
    \[
    \mathcal{J}\left(y,w\right)\approx
    \begin{cases}
    \cos\left(wy\right), & y>0,\\[3pt]
    0, & y<0,
    \end{cases}
    \qquad \left(\gamma\ \text{large}\right),
    \]
    Substituting these leading terms into the expression for $f_Y\left(y\right)$ gives
    \[
    \lim_{\gamma \to \infty}f^{\gamma}_Y\left(y\right)=
    \begin{cases}
    \displaystyle \frac{2}{\pi}\int_0^\infty e^{-c w^2}\cos\left(wy\right)dw, & y\ge 0,\\[10pt]
    0, & y<0.
    \end{cases}
    \]
    This matches the density of the folded normal random variable with mean $0$, variance $2c$, which appeared in our analysis for fixed $q$.
\end{proof}

\section{The proof of Theorem \ref{thm.main}}\label{exactformula}
The proof of Theorem \ref{thm.main} is divided into two steps. Before we dive into the technical parts, we give a road-map for the proof. In the first step, by using the tower property and recurrence relations, we write the expectation as a linear combination of the expectation depending only on $Q_1$, although the coefficients of each term is complicated. The second step is to simplify the coefficients. 

We start by introducing a lemma from \cite{martin2020stationary}. 

\begin{lemma}[\cite{martin2020stationary}, p.47 (6.6)]
    Let $\alpha \in \left(0,1\right)$ and $q \in [0,1)$, we have
    \begin{equation}
         \sum_{m = 0}^{\infty} \frac{\alpha^m q^{2m}}{\left(q;q\right)_{m}} = \left(1 -\alpha q\right)\sum_{m = 0}^{\infty}\frac{\alpha^m q^{m}}{\left(q;q\right)_{m}}.\label{martin}
    \end{equation}
\end{lemma}

\begin{lemma}
\begin{enumerate}[label=(\arabic*)]
Let $Q_n$ be the queuing process defined in \ref{construction} and set $c = x\left(1-x\right)$. Then:
\item For any $n \geq 1$,
    \begin{equation}
        \mathbb{E}_{x}\left[Q_{n+1} - Q_1\right] = c\sum_{i=1}^{n} \mathbb{E}_x\left[q^{Q_i}\right].\label{tower}
    \end{equation}
\item  For any integer $k\geq1$,
    \begin{equation}
        \mathbb{E}_x\left[q^{k Q_1}\right] = \left(q;q\right)_{k}.\label{q1}
    \end{equation}
\item For any integer $k \geq 1$,
    \begin{equation}
        \mathbb{E}_{x}\left[q^{k Q_{n+1}}\right] = \left(1 + c a_k\right) \mathbb{E}_{x}\left[q^{k Q_{n}}\right] + c b_k \mathbb{E}_{x}\left[q^{\left(k+1\right) Q_{n}}\right],\label{tower2}
    \end{equation}
    where 
    \[
    a_k = q^k + q^{-k} - 2, \quad b_k = 1 - q^{-k}.
    \]
\end{enumerate}
\end{lemma}
\begin{proof}
    \smallskip
    \noindent\textbf{(1)}  
    We first prove \eqref{tower}. This is a direct application of the tower property of expectation, which is used in the second equality below:
    \[
    \begin{aligned}
        \mathbb{E}_{x}\left[Q_{n+1} - Q_1\right] &= \sum_{i=1}^{n}\mathbb{E}_{x}\left[Q_{i+1} - Q_i\right] = \sum_{i=1}^{n}\mathbb{E}_{x}\left[\mathbb{E}\left[Q_{i+1} - Q_i| Q_i\right]\right]\\
        &=\sum_{i=1}^{n}\mathbb{E}_{x}\left[x\left(1-x\right)\left(1 - 1 + q^{Q_i}\right)\right]=c\sum_{i=1}^{n}\mathbb{E}_x\left[q^{Q_i}\right].
    \end{aligned}
    \]

    \smallskip
    \noindent\textbf{(2)}  
    Next, we establish \eqref{q1}.
    We begin with the case $k=1$. Using the recurrence relation of $\pi_k$, we compute
    \[
    \begin{aligned}
        \mathbb{E}_x\left[q^{Q_1}\right] = \sum_{m=0}^{\infty} q^m \pi_k = \sum_{m = 0}^{\infty}\frac{q^{2m}}{\left(q;q\right)_m} \pi_0 = \left(1-q\right)\sum_{m = 0}^{\infty}\frac{q^m}{\left(q;q\right)_m} \pi_0 = \left(1-q\right)\sum_{m = 0}^{\infty} \pi_k = 1-q.
    \end{aligned}
    \]
    We used \eqref{martin} with $\alpha = 1$ in the third equality.
    For general $k \geq 2$, plugging $\alpha = q^{k-1}$ into \eqref{martin} yields the recurrence
    \[
    \mathbb{E}_x\left[q^{k Q_1}\right] = \sum_{m=0}^{\infty} \frac{q^{\left(k+1\right)m}}{\left(q;q\right)_m}\pi_0 = \left(1 - q^k\right)\mathbb{E}_x\left[q^{\left(k-1\right) Q_1}\right].
    \]
    Therefore, \eqref{q1} follows by induction on $k$.

    \smallskip
    \noindent\textbf{(3)} 
    Finally, we prove \eqref{tower2}.  
    Once again, the key step is to apply the tower property of expectation:
    \[
    \begin{aligned}
        \mathbb{E}_{x}\left[q^{k Q_{n+1}}\right] &= \mathbb{E}_{x}\left[\mathbb{E}\left[q^{k Q_{n+1}}|Q_n\right]\right]\\
        &=\mathbb{E}_{x}\left[c q^{k\left(Q_n+1\right)} + c\left(1 - q^{Q_n}\right) q^{k\left(Q_n-1\right)} + \left(x^2 + \left(1-x\right)^2 + cq^{Q_N}\right)q^{k Q_n}\right]\\
        &= \mathbb{E}_{x}\left[\left(1 + c\left(q^k + q^{-k} -2\right)\right)q^{k Q_n} + c\left(1 - q^{-k}\right)q^{\left(k+1\right)Q_n}\right]\\
        &= \left(1 + c a_k\right) \mathbb{E}_{x}\left[q^{k Q_{n}}\right] + c b_k \mathbb{E}_{x}\left[q^{\left(k+1\right) Q_{n}}\right],
    \end{aligned}
    \]
    where $a_k = q^k + q^{-k} - 2$ and $b_k = 1 - q^{-k}$.
    
    \smallskip
    Combining the above three steps completes the proof of the lemma.
\end{proof}

\begin{proposition}
    We have 
    \begin{equation}
        \mathbb{E}_{x}\left[q^{Q_{n+1}}\right] = \sum_{k=0}^{n}c^k\left(q;q\right)_{k+1} \left(q^{-1};q^{-1}\right)_k h_{n-k}\left(1+ca_1,\dots 1+ca_{k+1}\right).\label{step1}
    \end{equation}
    As a consequence, 
    \begin{equation}
        \mathbb{E}_{x}\left[Q_{n+1} - Q_1\right] = c \sum_{i=0}^{n-1} \sum_{k=0}^{i}c^k\left(q;q\right)_{k+1} \left(q^{-1};q^{-1}\right)_k h_{i-k}\left(1+ca_1,\dots 1+ca_{k+1}\right).\label{step1.1}
    \end{equation}
\end{proposition}
\begin{proof}
    Combining \eqref{step1} and \eqref{tower} implies \eqref{step1.1}. So it is sufficient to prove  \eqref{step1}.
    
    We use induction to prove a stronger statement: For all $n\ge0$ and $k\ge1$,
    \begin{equation}
        \mathbb{E}_x\left[q^{kQ_{n+1}}\right]=\sum_{j=0}^{n} c^{j}\left(q;q\right)_{k+j}\Big(\prod_{m=0}^{j-1} b_{k+m}\Big)h_{n-j}\big(1+ca_k,1+ca_{k+1},\dots,1+ca_{k+j}\big).
        \label{strong}
    \end{equation}

    When $n=0$, the right-hand side is $\left(q;q\right)_kh_0\left(\cdot\right)=\left(q;q\right)_k=\mathbb{E}_x\left[q^{kQ_1}\right]$ due to \eqref{q1}.
    Assume \eqref{strong} holds for a fixed $n$ and all $k\ge1$. Then, by \eqref{tower2}
    \[
    \mathbb{E}_x\left[q^{kQ_{n+2}}\right] =\left(1+ca_k\right)\mathbb{E}_x\left[q^{kQ_{n+1}}\right] +cb_k\mathbb{E}_x\left[q^{\left(k+1\right)Q_{n+1}}\right].
    \]
    Apply the induction hypothesis to both expectations (first with parameter $k$, second with $k+1$), reindex $j\mapsto j+1$ in the second sum so that we can use the standard identity for homogeneous symmetric polynomials
    \begin{equation}
        h_m\left(x_k,\dots,x_{k+j}\right) = x_kh_{m-1}\left(x_k,\dots,x_{k+j}\right) + h_m\left(x_{k+1},\dots,x_{k+j}\right),\label{standard}
    \end{equation}
    with $x_r:=1+ca_r$. This merges the two sums into \eqref{strong} with $n$ replaced by $n+1$. Hence \eqref{strong} holds for all $n,k$ by induction.

    Finally, taking $k=1$ gives
    \[
    \mathbb{E}_x\left[q^{Q_{n+1}}\right] =\sum_{j=0}^{n} c^{j}\left(q;q\right)_{j+1}\Big(\prod_{m=1}^{j} b_{m}\Big) h_{n-j}\big(1+ca_1,\dots,1+ca_{j+1}\big),
    \]
    and, in particular, we have $\prod_{m=1}^j b_m=\left(q^{-1};q^{-1}\right)_j$. This is exactly the stated formula \eqref{step1}.
\end{proof}

Now we embark on the second step to derive a simplified formula. We need the so-called binomial formula of symmetric polynomials as input.
\begin{proposition}[\cite{macdonald1998symmetric}, p.47, example 10]
    Let $s_\lambda$ be the Schur polynomial indexed by a partition $\lambda$, we have
    \begin{equation}
        s_\lambda\left(1+x_1,\dots 1+ x_n\right) = \sum_{\mu \prec \lambda} det\left[\binom{\lambda_i + n -i}{\mu_j + n -i}\right]_{i,j = 1}^{n} s_{\mu}\left(x_1, \dots x_n\right).
    \end{equation}
    In particular, take $\lambda$ as a $n \times 1$ one row Young diagram, we have the binomial formula for homogeneous symmetric function.
    \begin{equation}
        h_k\left(1+x_1,\dots 1+ x_n\right) = \sum_{j = 0}^{k} \binom{k + N -i}{j + N -1} h_{j}\left(x_1, \dots x_n\right).\label{binomial}
    \end{equation}

\end{proposition}
Now we prove the main Theorem \ref{thm.main}. 
\begin{remark}
    In the following computation, we define $B_0\left(1,q\right) = 1$ for the simplicity of notation, this is not included in the original paper \cite{han1999q}.
\end{remark}
\begin{proof}[Proof of the theorem]
It is sufficient to prove, 
\begin{equation}
    \mathbb{E}_{x}\left[q^{ Q_{n+1}}\right] = \sum_{k=0}^{n} \left(1-q\right)^{2k+1} \left(-c\right)^{k} \binom{n}{k} B_k\left(1,q\right). \label{suff}
\end{equation}
Indeed, from \eqref{tower},
\[
\mathbb{E}_{x}\left[Q_{n+1}-Q_1\right]
= c\sum_{i=1}^{n}\mathbb{E}_{x}\big[q^{Q_i}\big].
\]
Apply \eqref{suff} with $n=i-1$ and and interchange finite sums:
\[
\mathbb{E}_{x}\left[Q_{n+1}-Q_1\right]
= c\sum_{k=0}^{n-1}\left(1-q\right)^{2k+1}\left(-c\right)^k B_k\left(1,q\right)
\sum_{i=k}^{n-1}\binom{i}{k}.
\]
Using the binomial identity $\sum_{i=k}^{n-1}\binom{i}{k}=\binom{n}{k+1}$, we obtain
\[
\mathbb{E}_{x}\left[Q_{n+1}-Q_1\right]
= c\sum_{k=0}^{n-1}\left(1-q\right)^{2k+1}\left(-c\right)^k\binom{n}{k+1}B_k\left(1,q\right)
= \sum_{k=0}^{n-1}\left(-1\right)^k\left(1-q\right)^{2k+1}c^{k+1}\binom{n}{k+1}B_k\left(1,q\right),
\]
which is \eqref{exact}.

Now we prove \eqref{suff}. We expand $h_{n-k}$ by using \eqref{binomial} and homogeneity of $h_{n-k}$ to obtain
\[
h_{n-k}\big(1+c a_1,\dots,1+c a_{k+1}\big)
=\sum_{j=0}^{n-k} c^{j}\binom{n}{k+j}h_{j}\left(a_1,\dots,a_{k+1}\right).
\]
Substituting $h_{n-k}\big(1+c a_1,\dots,1+c a_{k+1}\big)$ in \eqref{step1} by the previous equality and collecting powers of $c$, we get
\[
\mathbb{E}_{x}\big[q^{Q_{n+1}}\big]
= \sum_{k=0}^{n}\sum_{j=0}^{n-k}
   c^{k+j}\binom{n}{k+j}
   h_{j}\left(a_1,\dots,a_{k+1}\right)
   \left(q;q\right)_{k+1}\left(q^{-1};q^{-1}\right)_{k}.
\]
With the change of variables $m=k+j$ (hence $k\le m\le n$), this becomes
\[
\mathbb{E}_{x}\big[q^{Q_{n+1}}\big]
= \sum_{k=0}^{n}\sum_{m=k}^{n}
   c^{m}\binom{n}{m}
   h_{m-k}\left(a_1,\dots,a_{k+1}\right)
   \left(q;q\right)_{k+1}\left(q^{-1};q^{-1}\right)_{k}.
\]
Exchanging the sums yields
\begin{equation}
    \mathbb{E}_{x}\big[q^{Q_{n+1}}\big] = \sum_{m=0}^{n} c^{m}\binom{n}{m}\sum_{k=0}^{m}h_{m-k}\left(a_1,\dots,a_{k+1}\right)\left(q;q\right)_{k+1}\left(q^{-1};q^{-1}\right)_{k}.
    \label{step2}
\end{equation}

By comparing the coefficient of $c^m$ in \eqref{suff} and in \eqref{step2}, it is sufficient to show,
\begin{equation}
    \sum_{k=0}^{m} h_{m-k}\left(a_1,\dots a_{k+1}\right)\left(q;q\right)_{k+1}\left(q^{-1};q^{-1}\right)_k = \left(1-q\right)^{2m+1}\left(-1\right)^m B_m\left(1,q\right).\label{compare}
\end{equation}

Applying \eqref{standard} to the left hand side of \eqref{compare} and reindexing, we obtain
\[
LHS=\sum_{k=0}^{m-1} h_{m-k-1}\left(a_1,\dots,a_{k+1}\right)\left(q;q\right)_{k+1}\left(q^{-1};q^{-1}\right)_k
\bigl[a_{k+1}+\left(1-q^{k+2}\right)\left(1-q^{-k-1}\right)\bigr].
\]
Since \(a_{k+1}=q^{k+1}+q^{-k-1}-2\), the bracket simplifies to
\[
a_{k+1}+\left(1-q^{k+2}\right)\left(1-q^{-k-1}\right)=\left(1-q\right)\left(q^{k+1}-1\right)=-\left(1-q\right)\left(1-q^{k+1}\right),
\]
hence
\[
LHS=\sum_{k=0}^{m-1}-\left(1-q\right)\left(1-q^{k+1}\right)h_{m-k-1}\left(a_1,\dots,a_{k+1}\right)\left(q;q\right)_{k+1}\left(q^{-1};q^{-1}\right)_k.
\]
Now use \(\left(q^{-1};q^{-1}\right)_k=\left(-1\right)^k q^{-k\left(k+1\right)/2}\left(q;q\right)_k\) and \(\left(q;q\right)_{k+1}=\left(1-q^{k+1}\right)\left(q;q\right)_k\) to get
\[
\left(1-q^{k+1}\right)\left(q;q\right)_{k+1}\left(q^{-1};q^{-1}\right)_k=\left(-1\right)^k q^{-k\left(k+1\right)/2}\left(q;q\right)_{k+1}^2,
\]
so
\[
LHS=\sum_{k=0}^{m-1}\left(-1\right)^{k+1}\left(1-q\right)q^{-k\left(k+1\right)/2}h_{m-k-1}\left(a_1,\dots,a_{k+1}\right)\left(q;q\right)_{k+1}^2.
\]

Finally, note that
\[
a_i=q^i+q^{-i}-2=q^{-i}\left(1-q^i\right)^2=\left(1-q\right)^2q^{-i}[i]_q^{2},
\]
so by the homogeneity of \(h_{m-k-1}\) and \(\left(q;q\right)_{k+1}=\left(1-q\right)^{k+1}\left[k+1\right]_q!\), the total \(\left(1-q\right)\)-power is
\(\left(1-q\right)^{2\left(m-k-1\right)}\cdot\left(1-q\right)^{2\left(k+1\right)}\cdot\left(1-q\right)=\left(1-q\right)^{2m+1}\). Therefore
\[
\begin{aligned}
LHS &=\left(1-q\right)^{2m+1}\sum_{k=0}^{m-1}\left(-1\right)^{k+1}q^{-k\left(k+1\right)/2}\left([k+1]_q!\right)^2 \\
&\qquad\qquad\qquad\qquad\times
h_{m-k-1}\bigl(q^{-1}\left[1\right]_q^{2},q^{-2}\left[2\right]_q^{2},\dots,q^{-\left(k+1\right)}\left[k+1\right]_q^{2}\bigr).
\end{aligned}
\]

This completes the simplification of the left-hand side. This is equal to the right hand side of \eqref{compare} due to \eqref{alter} and the fact that $\left(-1\right)^{k+1}=\left(-1\right)^{-k-1}$.
\end{proof}

\begin{remark}
    What we seek from the exact formula is the asymptotic size of the convoy. In the TASEP case, the expression admits a substantial simplification: by exploiting the combinatorial interpretation of the $q$-Genocchi numbers, it reduces to a hypergeometric function whose limiting behavior is well understood. We present this result to close the section. For the general $q$-case, however, a deeper understanding of this combinatorial interpretation would be required, and pursuing that direction lies beyond the scope of the present work.
\end{remark}

\begin{theorem}\label{q0genocchi}
    The q-Genocchi number is the Catalan number when $q = 0$. For any $k\geq0$, we have
    \[
    B_k\left(1,0\right) = \frac{1}{1+k}\binom{2k}{k}.
    \]
\end{theorem}
\begin{proof}
    To prove this lemma, we introduce a combinatorial interpretation of the q-Genocchi number. 

\begin{definition}
    A surjective pistol $p$ of size $2n$ is a surjective map from $\{1, 2, \dots ,2n \}$ on to $\{2, 4, \dots ,2n \}$, such that,
    \begin{itemize}
        \item $p$ is a surjection,
        \item $\forall i \in \{1, 2, \dots 2n \}, p\left(i\right)\geq i.$
    \end{itemize}
    We denote by $\mathcal{P}_{2n}$ the set of surjective pistols of size $2n$. For a surjective pistol $p \in \mathcal{P}_{2n}$, a special inversion is a pair $\left(i,j\right)$ such that 
    \begin{itemize}
        \item $i > j$, 
        \item $p\left(i\right) < p\left(j\right)$, and 
        \item $i$ is the rightmost position attaining the value $p\left(i\right)$.
    \end{itemize}
    We denote by $\operatorname{sinv}\left(p\right)$ the number of special inversions of $p$.
\end{definition}

The following theorem characterize the q-Genocchi number by surjective pistols.

\begin{theorem}[\cite{vong2014combinatoire}, Theorem 3.2.4]
    The $q$-Genocchi number $B_n\left(1,q\right)$ admits the following combinatorial interpretation in terms of surjective pistols:
    \[
    B_n\left(1,q\right) = \sum_{p \in \mathcal{P}_{2n}} q^{sinv\left(p\right)}.
    \]
\end{theorem}

A direct corollary is that $B_n\left(1,0\right)$ is equal to the number of surjective pistols such that $sinv\left(p\right) = 0$. Moreover, the condition $\operatorname{inv}\left(p\right)=0$ is equivalent to $\operatorname{sinv}\left(p\right)=0$. Indeed, if a surjective pistol $p$ contains an inversion $\left(i,j\right)$, then let $j'$ be the rightmost position such that $p\left(j'\right)=p\left(j\right)$. By construction, the pair $\left(i,j'\right)$ is a special inversion. Thus every ordinary inversion gives rise to a special inversion, and conversely if $\operatorname{sinv}\left(p\right)=0$ then no inversion can occur. $\operatorname{inv}\left(p\right)=0$ implies $\operatorname{sinv}\left(p\right)=0$ is clear.

The following lemma ends the proof of the theorem.
\begin{lemma}
    The number of surjective pistols of size $2n$ such that $inv\left(p\right) = 0$ is the Catalan number.
\end{lemma}
\begin{proof}
Let $p$ be such a pistol. For each $i \in \{1,\dots,n\}$, let $a_i$ denote the multiplicity of $2i$ in the image of $p$. The sequence $\left(a_1,\dots,a_n\right)$ satisfies the following conditions:

\begin{itemize}
    \item $a_i \geq 1$ for all $i$, since $p$ is surjective onto $\{2,4,\dots,2n\}$.
    \item For every $j$, we have $\sum_{i=1}^j a_i \leq 2j$, because at least $2n-2j$ of the preimages must map to values greater than $2j$, by the condition $p\left(k\right) \geq k$.
    \item Finally, $\sum_{i=1}^n a_i = 2n$, since the domain of $p$ has size $2n$.
\end{itemize}

We note that the sequence $\left(a_1,\dots,a_n\right)$ encodes all the surjective pistols $p$ when $\operatorname{inv}\left(p\right)=0$, since the non-inversion condition forces the values of $p$ to appear in increasing blocks (for example, of the form $2,4,4,4,6,6$).  

Now define $b_i = 2-a_i$. Then $\left(b_1,\dots,b_n\right)$ satisfies
\[
b_i \leq 1 \quad \text{for all } i, 
\qquad 
\sum_{k=1}^j b_k \geq 0 \quad \text{for all } j, 
\qquad 
\sum_{i=1}^n b_i = 0.
\]

These are precisely the conditions that characterize Dyck paths of semilength $n$: if we interpret $b_i=1$ as an up-step and $b_i=-1$ as a down-step, the above conditions say the path never goes below the axis and ends at height $0$. By a well-known result (see \cite{Stanley_Fomin_1999} for example), the number of such sequences is the $n$th Catalan number $C_n$.

Hence the number of surjective pistols of size $2n$ with $\operatorname{inv}\left(p\right)=0$ is $C_n$, as claimed.
\end{proof}
\end{proof}

Now plugging in the Catalan number into Theorem~\ref{thm.main}, we obtain
\[
\mathbb{E}_x\left[\mathcal{C}_0\right]
= \mathbb{E}_x\left[Q_{n+1}-Q_1\right]
= \sum_{k=0}^{n-1} \left(-1\right)^k c^{k+1}\frac{1}{k+1}\binom{n}{k+1}\binom{2k}{k}.
\]
To evaluate the sum, set $r=k+1$. Using
\[
\binom{2r-2}{r-1}
= 4^{r-1}\frac{\left(1/2\right)_{r-1}}{\left(r-1\right)!}, 
\qquad
\binom{n}{r} = \frac{\left(-n\right)_r}{r!}\left(-1\right)^r,
\]
together with
\(
\frac{1}{r} = \frac{\left(1\right)_r}{\left(2\right)_r}, 
\)
the summand becomes
\[
\left(-1\right)^k c^{k+1}\frac{1}{k+1}\binom{n}{k+1}\binom{2k}{k}
= \frac{\left(4c\right)^r}{2}\frac{\left(-n\right)_r\left(-\tfrac12\right)_r}{\left(1\right)_rr!}.
\]
Summing over $r=1,\dots,n$ and recognizing the truncated hypergeometric expansion gives
\[
\sum_{k=0}^{n-1} \left(-1\right)^k c^{k+1}\frac{1}{k+1}\binom{n}{k+1}\binom{2k}{k}
= \tfrac12\left({}_2F_1\left(-n,-\tfrac12;1;4c\right)-1\right).
\]
Therefore,
\[
\mathbb{E}_x\left[\mathcal{C}_0\right]
= \frac{1}{2}\left({}_2F_1\left(-n,-\tfrac12;1;4c\right)-1\right).
\]

The asymptotic behavior of hypergeometric functions has been intensively studied. We require the following result from \cite{bateman1953higher}. 

\begin{theorem}[\cite{bateman1953higher}, Chapter II,2.3.2]
    Consider a sequence of Gaussian hypergeometric functions $_2F_1\left(a,b_n;d;z\right)$, where $a, d, z$ are all fixed and $|b_n| \to \infty$,  if $d$ is not a non-positive integer, $|z| < 1$ and $arg \left(b_n z\right) \in \left(\frac{-3\pi}{2}, \frac{\pi}{2}\right)$, then We have
    \begin{equation}
        _2F_1\left(a,b_n;d;z\right) = e^{-i\pi a}\left[ \frac{\Gamma\left(d\right)}{\Gamma\left(d - a\right)}\right]\left(b_n z\right)^{-a} \left[1 + O\left(|b_n z|^{-1}\right)\right] + \left[ \frac{\Gamma\left(d\right)}{\Gamma\left(a\right)}\right] e^{b_n z} \left(b_n z\right)^{a-d} \left[1 + O\left(|bz|^{-1}\right)\right].\label{asymp}
    \end{equation}
\end{theorem}

\begin{theorem}
    As $n\to\infty$, the expected convoy size in TASEP satisfies
    \[
    \lim_{n\to\infty}\frac{\mathbb{E}_x\left[\mathcal{C}_0\right]}{\sqrt{n}} = \sqrt{\frac{4x\left(1-x\right)}{\pi}}.
    \]
\end{theorem}
\begin{remark}
    This provides an alternative proof to Theorem \ref{thm.uni} when $q = 0$. It would be interesting to find a similar proof for the general $q \in [0,1)$. Recall that we proved in Theorem \ref{q0genocchi} the coefficient of the constant term of $B_k\left(1,q\right)$ (as a polynomial with respect to $q$) is $\frac{1}{1+k}\binom{2k}{k}$. One can further investigate $\left[q^n\right] B_k\left(1,q\right)$. For instance,
    \[
    \left[q^1\right] B_k\left(1,q\right) = \binom{2k+1}{k-2}, \text{when} \quad k\geq 2, \qquad [q^2] B_k\left(1,q\right) = \left(k+1\right) \binom{2k+1}{k-3}, \text{when} \quad k\geq 3.
    \]
    If one can find general expressions for $\left[q^n\right] B_k\left(1,q\right)$, the universal result might be proved by a similar computation and asymptotic behavior of the hypergeometric functions. However, we cannot find formulas for $\left[q^n\right] B_k\left(1,q\right)$ when $n \geq 4$.
\end{remark}

\begin{proof}
From the previous identity we have
\[
\mathbb{E}_x\left[\mathcal{C}_0\right]
= \tfrac{1}{2}\bigl({}_2F_1\left(-n,-\tfrac12;1;4c\right)-1\bigr),
\qquad c=x\left(1-x\right).
\]
Applying \eqref{asymp} with $a=-\tfrac12$, $b_n=-n$, $d=1$, and $z=4c$, and using $\Gamma\left(1\right)=1$, $\Gamma\left(\tfrac32\right)=\tfrac{\sqrt{\pi}}{2}$, we obtain
\[
{}_2F_1\left(-n,-\tfrac12;1;4c\right)
\sim \frac{2}{\sqrt{\pi}}\left(4cn\right)^{1/2},\qquad n\to\infty.
\]
Thus
\[
\mathbb{E}_x\left[\mathcal{C}_0\right]\sim \sqrt{\tfrac{4cn}{\pi}},\qquad
\lim_{n\to\infty}\frac{\mathbb{E}_x\left[\mathcal{C}_0\right]}{\sqrt{n}}
=\sqrt{\tfrac{4x\left(1-x\right)}{\pi}}.
\]
\end{proof}

\bibliographystyle{plain}
\bibliography{ref}
\section*{}

\noindent
(\textsc{Yuan Tian})\\
Max Planck Institute for Mathematics in the Sciences, Germany\\
Universität Leipzig, Germany\\
Email address: \texttt{yuan.tian@mis.mpg.de}

\end{document}